\documentclass[11pt]{amsart}
\usepackage{amsmath,amsthm,amssymb,latexsym,graphicx,verbatim,enumerate,
ifpdf,mdwlist,xspace}
\usepackage{mathabx}
\ifpdf
\usepackage{hyperref}
\else
\usepackage[hypertex]{hyperref}
\fi
\usepackage{color}
\usepackage{amssymb, amsmath, amsthm}
\usepackage{graphicx}
\usepackage{cancel}
\usepackage{float}

\usepackage{multicol}
\usepackage{lipsum}

\providecommand{\B}{\mathcal{B}}

\newtheorem{question}{Question}

\newtheorem{thm}{Theorem}[section]
\newtheorem{definition}[thm]{Definition}

\newtheorem{remark}[thm]{Remark}

\newtheorem{lemma}[thm]{Lemma}

\newtheorem{prop}[thm]{Proposition}

\newtheorem{claim}{Claim}[section]

\newenvironment{defn}{\begin{definition} \rm}{ \end{definition}}

\newcommand{\skipmm}[1]{\hspace{-#1mm}}

\DeclareMathOperator{\A}{\mathcal{A}}

\renewcommand{\phi}{\varphi}

\DeclareMathOperator{\Inf}{\mathbf{Inf}}

\newcommand{\Ex}{\mathbf{Ex}}

\newcommand{\rel}[1]{\mathrel{#1}}

\title[Learning  structures and Borel equivalence relations]{Learning algebraic  structures\\ with the help of Borel equivalence relations}

\date{\today}

\thanks{\emph{Acknowledgements.} Bazhenov was supported by Mathematical Center in Akademgorodok under agreement No. 075-15-2019-1613 with the Ministry of Science and Higher Education of the Russian Federation. Cipriani's research was partially supported by the Italian PRIN 2017 Grant ``Mathematical Logic: models, sets, computability''.}

\author[N.~Bazhenov]{Nikolay Bazhenov}
\address{Sobolev Institute of Mathematics\\ 4 Acad. Koptyug Ave., Novosibirsk,
630090, Russia}
\email{{bazhenov@math.nsc.ru}}

\author[V.~Cipriani]{Vittorio Cipriani}
\address{Department of Mathematics, Computer Science and Physics, University of Udine, Italy}
\email{{cipriani.vittorio@spes.uniud.it}}

\author[L.~San Mauro]{Luca San Mauro}
\address{Department of Mathematics, Sapienza University of Rome, Italy}
\email{{luca.sanmauro@uniroma1.it}}



\begin{document}

\begin{abstract}
We study algorithmic learning of algebraic structures. In our framework, a
learner receives larger and larger pieces of an arbitrary copy of a computable
structure and, at each stage, is required to output a conjecture about the isomorphism type of such a structure. The learning is successful if the conjectures
eventually stabilize to a correct guess. We prove that a family of structures is
learnable if and only if its learning domain is continuously reducible to
the relation $E_0$ of eventual agreement on reals. This motivates a novel research program, that is, using
descriptive set theoretic tools to calibrate  the
(learning) complexity of nonlearnable families. Here, we focus on the learning power of well-known
benchmark Borel equivalence relations (i.e., $E_1$, $E_2$, $E_3$, $Z_0$, and $E_{set}$).
\end{abstract}

\maketitle

\section{Introduction}

This paper wishes to connect two seemingly distant areas of research: algorithmic learning theory and the theory of Borel equivalence relations.

\smallskip

Algorithmic learning theory dates back to the work of Gold~\cite{Gold67} and Putnam~\cite{putnam1965trial} in the 1960s and it comprehends several formal frameworks for the inductive inference. Broadly construed, this research program models the ways in which a \emph{learner} can achieve systematic
knowledge about a given \emph{environment}, by accessing to more and more data about it. Although in classical paradigms the objects to be inferred are either formal languages or recursive functions (see, e.g., \cite{zz-tcs-08,lange2008learning}), in  recent times there has been a growing interest in 
the learning of data embodied with a structural content,
with special attention  paid to familiar classes of algebraic structures, such as vector spaces, rings, trees, and matroids \cite{HaSt07,SV01,MS04,GaoStephan-12}. 

In previous works~\cite{FKS-19,bazhenov2020learning}, relying on ideas and technology from computable structure theory, we introduced and explored our own framework. Intuitively (formal details will be given below), an agent receives larger and larger pieces of an arbitrary copy of a computable structure and, at each stage, is required to output a conjecture about the isomorphism type of such a structure. Then, the learning is successful if the conjectures eventually stabilize to a correct guess. 
See also \cite{Gly:j:85,Mar-Osh:b:98} for other frameworks which can be similarly applied to arbitrary structures.

As single countable structures can always be learned, the emphasis of our research is on the learnability (or lack thereof) of families of structures. 
 In \cite{bazhenov2020learning}, by adopting infinitary logic, we obtained a complete model theoretic characterization of which families of algebraic structures are learnable. From such a characterization, it immediately follows  that some seemingly innocent learning problems are out of reach: e.g., no agent can  learn whether the observed structure is a copy of the isomorphism type of the natural numbers or of the integers, that is, the pair of linear orders $\{\omega, \zeta\}$ is nonlearnable. We have also addressed the question of how much computational power is needed to handle a given learning problem: in \cite{bazhenov2021turing}, we constructed a pair of structures which is learnable, but no Turing machine can learn it.

\smallskip

A defect of our framework has been that, until this day, we had no way of calibrating the complexity of nonlearnable families. The present paper aims at rectifying this situation, by offering a new hierarchy to classify the complexity of learning problems for algebraic structures. To this end, we borrow several ideas from descriptive set theory. This is readily justified. Indeed, a primary theme of modern descriptive set theory is the study of the complexity of equivalence relations defined on suitable topological spaces, with a special focus on the so-called \emph{Borel equivalence relations}, to be defined below (see, e.g., \cite{gao2008invariant,kanoveui,hjorth2010borel}). A popular way of evaluating the complexity of Borel equivalence relations is by defining an appropriate reducibility: in general, a reduction from an
equivalence relation $E$ on $X$ to an equivalence relation $F$ on $Y$ is a (nice)
function $f : X \to Y$ which induces an embedding on the equivalence
classes, $X_{/E} \to Y_{/F}$.

A large body of literature, within the theory of Borel equivalence relations, concerns  equivalence relations associated to \emph{classification problems}, i.e., problems which ask to scaffold a given family of mathematical structures up to certain notion of ``similarity''. Crucially to our interests, isomorphism problems form an important subclass of classification problems, and descriptive set theorists have put serious effort in ranking the complexity of isomorphism problems for various familiar classes of countable structures (such as groups, trees, linear orderings, and Boolean algebras \cite{FriedmanStanley,mekler1981stability,
camerlo2001completeness}).

The above description, albeit necessarily brief and oversimplified, may resound with our learning framework. Indeed, in our paradigm the learner is required to guess the isomorphism type for each structure from the family to be learned. Hence, the nonlearnability of a certain family $\mathfrak{K}$ of algebraic structures is, in a sense,  rooted in the complexity of the isomorphism relation associated with $\mathfrak{K}$. Yet, two  aspects shall be  stressed:

\begin{enumerate}
\item The isomorphism relations customarily studied in descriptive set theory refer to \emph{large} collections of countable structures (e.g., \emph{all} graphs, abe\-lian groups, or metric spaces). On the contrary, here we focus on learning \emph{small} families (i.e., countable families, and in fact often finite ones as in \cite{bazhenov2021turing}); 
\item At any finite stage, the learner sees 
only a finite fragment of the structure to be learned, and each conjecture must  be formulated without knowing how the observed structure will be extended. In topological terms, this coincides with asking that the learning must be a \emph{continuous} process. 
\end{enumerate}

These observations are clearly informal. But, in Section~3, we'll be able to make them precise, while offering a new characterization of learnability, this time from a descriptive set theoretic point of view. Namely, we'll show that \emph{a family of structures $\mathfrak{K}$ is learnable if and only the isomorphism relation  associated with $\mathfrak{K}$ is continuously reducible to the relation $E_0$ of eventual agreement on reals} (Theorem~\ref{from E_0 to InfEx}). As the relation $E_0$ is a fundamental benchmark in the theory of Borel equivalence relations (e.g., the celebrated Glimm-Effros dichotomy states that $E_0$ is the successor of the identity on reals within the Borel hierarchy~\cite{HKL}), such a new characterization of learnability for structures may serve as a piece of evidence  that our paradigm is a natural one. 

Furthermore, by replacing $E_0$ with Borel equivalence relations of higher complexity, one  immediately unlocks  the promised  hierarchy of learning problems. That is, we'll say that a family of structures $\mathfrak{K}$ is \emph{$E$-learnable}, for a Borel equivalence relation $E$, if there is a continuous reduction from the isomorphism relation  associated with $\mathfrak{K}$ to $E$.  Then, Sections~4--6 are dedicated to an investigation of the learning power of several benchmark Borel equivalence relations, offering both examples of relations which do not enlarge the scope of $E_0$-learnability (Theorems~\ref{prop:E_01} and \ref{thm:E_02}) and  equivalence relations which do so (Theorem~\ref{thm:E_03_for_countable} and \ref{thm:E_3-characterization}). Interestingly, we'll show that the learning power of some equivalence relations is affected by whether we restrict the attention to families containing only finitely many isomorphism types, or we rather allow countably infinite families. The final section contains a brief description of our intended future research in this area.

\section{Preliminaries} 
As this paper is at the crossroad of a number of areas -- namely, computable structure theory, algorithmic learning theory, and descriptive set theory -- it will be convenient to break down these preliminaries in multiples subsections. We assume that the reader has some basic knowledge of topology and computability, as it can be found in \cite{soare2016turing}. In particular, by $(\phi_e)_{e\in\omega}$, $(W_e)_{e\in\omega}$, and $(\Phi^X_e)_{e\in\omega}$ we denote a uniformly
computable list of, respectively, all partial computable functions, all computably
enumerable (c.e.) sets, and all Turing operators with oracle $X$. 

\subsection*{Reals}
We adopt the common habit of calling  infinite binary sequences \emph{reals}. To distinguish them from the natural numbers, reals are denoted by lowercase Greek letters (e.g., $\alpha,\beta$). Functions on reals are denoted by uppercase Greek letters (e.g., $\Gamma,\Psi$). The $m$-th binary digit of a real $\alpha$ is denoted by $\alpha(m)$. By $\alpha^{[m]}$, we denote the real representing the $m$-th column $\alpha(\langle m,\cdot\rangle)$ of $\alpha$. The \emph{symmetric difference} $\alpha\triangle \beta$ of two reals is defined in the usual way:
\[
	(\alpha \triangle \beta)(i) = 1\ \Leftrightarrow\ \alpha(i) \neq \beta(i).
\]

\subsection*{The Cantor space}

In this paper, we focus on equivalence relations defined on the \emph{Cantor space}. Such a space, written as $2^\omega$, can be represented as  the collection of reals, equipped with the product topology of the discrete topology on $\{0,1\}$. 
For a binary string $\sigma$, the \emph{cylinder} $[\sigma]$ is defined as the collection of reals extending $\sigma$, i.e.,
\[
[\sigma]:=\{ \alpha\in 2^\omega : \sigma\subseteq \alpha \}.
\]
 These cylinders form a basis  of $2^\omega$. A subset $X$ of Cantor space is \emph{Borel}, if it  can be constructed by open sets, taking countable unions, countable intersections, and complements. 
 A function $\Gamma : 2^\omega \to 2^\omega$  is \emph{Borel}, if the preimage of any Borel set is Borel; it is \emph{continuous}, if the preimage $\Gamma$ of any open set is open. A Turing operator $\Phi$ can be naturally regarded as a partial function $\Phi: 2^\omega\to 2^\omega$,  where $\Phi(\alpha)$ is defined if and only $\Phi^\alpha$ is total. Note that  this partial function is continuous, since  converging oracle computations are always determined by a  finite initial segment of the oracle. For a function $\Gamma : 2^\omega \to 2^\omega$, a real $\alpha$, and a number $s\in\omega$, the notation $\Gamma(\alpha)(s)$ refers to the $s$th bit of $\Gamma(\alpha)$.

  Throughout the paper, we will often  rely on the following lemma which expresses that every continuous function is computable with respect to some powerful enough oracle.
 
 \begin{lemma}[folklore]\label{continuous is computable with oracle}
 If $\Gamma: 2^\omega \to 2^\omega$ is continuous, then there are an oracle $X$ and Turing operator $\Phi$ so that
 \[
 \Gamma(\alpha)=\Phi^{X\oplus \alpha},
 \]
 for every real $\alpha$.
 \end{lemma}
 
 \begin{proof}
 The continuity of $\Gamma$ guarantees that there is a  function $h: 2^{<\omega}\to 2^{<\omega}$ which satisfies the following requirements:
 	\begin{enumerate}
 	\item for $\sigma, \tau \in 2^{<\omega}$, if $\sigma \subseteq \tau$, then $h(\sigma)\subseteq h(\tau)$;
 	\item for all $\alpha \in 2^\omega$, $\Gamma(\alpha)=\bigcup_{\sigma \subset \alpha} h(\sigma)$.
 	\end{enumerate}
Hence, it is straightforward to define the desired Turing operator by choosing an oracle $X$ that computes $h$.
 \end{proof}
 
 In this paper, it is convenient to call every Turing operator of the form $\Phi^{X\oplus\alpha}$ a  \emph{Turing $X$-operator}. Intuitively,  Turing $X$-operators can be identified with a Turing machine which has three tapes: the input tape (on which the machine is allowed to work), the
output tape, and the oracle tape, where the oracle tape always contains the characteristic function of $X$. 

\subsection*{Benchmark Borel equivalence relations} To evaluate the complexity of equivalence relations on reals, one defines a suitable reducibility. Let $E$ and $F$ be equivalence relations on $2^\omega$. A \emph{reduction} from $E$ to $F$ is a function $\Gamma: 2^\omega \to 2^\omega$ such that
\[
\alpha \rel{E} \beta \Leftrightarrow {\Gamma(\alpha) \rel{F} \Gamma(\beta)},
\]
for all reals $\alpha,\beta$. It is common to impose definability requirements on the functions inducing a reduction. \emph{Borel reductions}, introduced in \cite{FriedmanStanley}, are regarded as the most useful tools for calculating the relative complexity of equivalence relations. But in this paper, we'll concentrate on the following stronger reducibility: $E$ is \emph{continuously reducible} to $F$, if  there is a continuous function $\Gamma: 2^\omega \to 2^\omega$ which reduces $E$ to $F$.

The following combinatorial equivalence relations on reals are widely considered in descriptive set theory as benchmarks to gauge the complexity of natural classification problems (see, e.g., \cite{kanoveui}):

\begin{itemize}
\smallskip

\item[(a)] $\alpha \rel{E_0} \beta$ if and only if
\[
(\exists m)(\forall n\geq m)(\alpha (n)=\beta(n)).
\]
	\item[(b)] $\alpha\rel{E_1}\beta$ if and only if
	\[
		(\forall^{\infty} m\in\omega) (\alpha^{[m]} = \beta^{[m]}).
	\]
	
	\item[(c)] $\alpha \ E_2\ \beta$ if and only if
	\[
		\sum_{k=0}^{\infty} \frac{(\alpha\triangle\beta)(k)}{k+1}\ <\ \infty.
	\]
	
	\item[(d)] $\alpha \ E_3\ \beta$ if and only if 
\[	
	(\forall m )( \alpha^{[m]} \ E_0 \ \beta^{[m]}).
\]
	
	\item[(e)] $\alpha\ E_{set}\ \beta$ if and only if 
\[	
	\{ \alpha^{[m]}\,\colon m\in\omega \} = \{ \beta^{[m]}\,\colon m\in\omega\}.
	\]
	
	\item[(f)] 
	$\alpha \ Z_0\ \beta$ if and only if $\alpha\triangle\beta$ has (asymptotic) density zero, i.e.
	\[
		\lim_{k\to \infty} \frac{\mathrm{card}(\{ i\leq k\,\colon \alpha\triangle\beta(i) = 1\})}{k+1} = 0.
	\]
\end{itemize}
These benchmark equivalence relations lie at the base of the Borel hierarchy: Figure~\ref{fig:benchmark}, which is taken from \cite{coskey2012hierarchy},  shows all continuous reducibilities between them (in fact, the diagram is the same even if we restrict to computable reductions). For more background about Borel, continuous, and computable reductions, see \cite{gao2008invariant,hjorth2010borel,miller2021computable,bazhenov2021computational}.

 \begin{figure}
\begin{center}
 \includegraphics[scale=1]{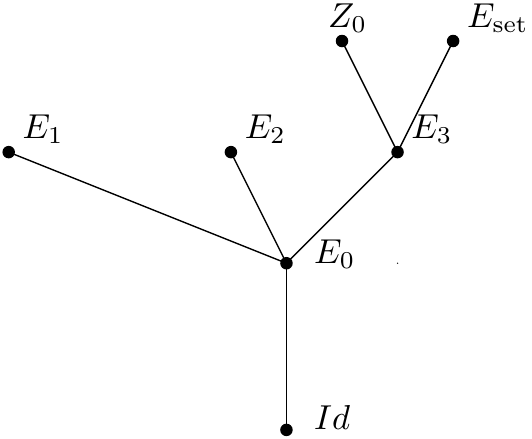}
\end{center}
\caption{Reductions up to continuous reducibility.}
\label{fig:benchmark}
 \end{figure}

\subsection*{Computable structures.}  A \emph{signature} $L$ lists all function symbols and relation symbols which characterize an algebraic structure. All our structures have  domain $\omega$. We say that two structures are \emph{copies} of each other if they are isomorphic. 


In computable structure theory, one measures the complexity of an $L$-stru\-cture $\mathcal{A}$ by identifying $\mathcal{A}$ with its \emph{atomic diagram},  i.e., the collection of atomic formulas  which are true of $\mathcal{A}$. Up to a suitable G\"odel numbering of $L$-formulas, the atomic diagram of $\mathcal{A}$ may be regarded as a real: this provides a natural way of assigning to each structure a Turing degree $\mathbf{d}$, representing its algorithmic complexity. 
Any computable structure $\mathcal{A}$ in a relational signature (i.e., with no function symbols) can be presented
as an increasing union of its finite substructures
\[
\A\restriction_0 \ \subseteq \A\restriction_1 \ \subseteq  \ldots \subseteq \A\restriction_i \ \subseteq \ldots,
\]
where $\A\restriction_n$ denotes the restriction of $\A$ to the domain $\{0,1,\ldots,n\}$  and $\A=\bigcup \mathcal{A} \restriction_i$.  For more background about computable structures, see \cite{AK00,EG-00}.

\subsection*{Infinitary formulas.}
To assess the model theoretic complexity of countable structures, it is common to work in the infinitary logic $\mathcal{L}_{\omega_1\omega}$, which allows to take the conjunctions or disjunctions of infinite sets of formulas. In particular, \emph{infinitary $\Sigma_n$ formulas} are defined as follows,
\smallskip

\begin{itemize}
	\item $\Sigma^{\inf}_0$  and $\Pi^{\inf}_0$ formulas are quantifier-free first-order formulas.
	
	\item A $\Sigma_{n+1}^{\inf}$ formula $\psi(\bar{x})$ is a countably infinite disjunction
	\[
		\underset{i\in I}{\bigvee\skipmm{6.8}\bigvee}\ \exists \bar y_i \xi_i(\bar x, \bar y_i),
	\]
	where each $\xi_i$ is a $\Pi^{\inf}_{n}$ formula.
	
	\item A $\Pi_{n+1}^{\inf}$ formula $\psi(\bar x)$ is a countably infinite conjunction
	\[
		\underset{i\in I}{\bigwedge\skipmm{6.8}\bigwedge}\ \forall \bar y_i \xi_i(\bar x, \bar y_i),
	\]
	where each $\xi_i$ is a $\Sigma^{\inf}_{n}$ formula.
\end{itemize}
Next, \emph{computable infinitary $\Sigma_n$ formulas} (or \emph{$\Sigma^c_n$ formulas}, for short) are defined in the same way as above, but requiring infinite conjunctions and disjunctions to range over c.e.\ sets of  (computable) formulas. Finally, computable infinitary formulas can be relativized to an arbitrary oracle $X$: the class of \emph{$X$-computable infinitary $\Sigma_n$~formulas} is denoted by $\Sigma^c_n(X)$. For more background about infinitary formulas, see~\cite{marker2016lectures}.

\subsection{Our framework} We shall now revisit  
the learning framework  presented in \cite{bazhenov2020learning}.  Our exposition  closely follows \cite{bazhenov2021turing}. In particular, we ignore how a given family is enumerated and we just assume that any structure $\A$ gives rise to a corresponding   \emph{conjecture} $\ulcorner \A\urcorner$, to be understood as conveying the piece of information ``this is $\A$''. 

\begin{defn}
Suppose that $\mathbf{P}$ is the learning problem associated to  a countable family $\mathfrak{K}$ of nonisomorphic computable structures. The ingredients of our framework may be specified as follows. For $\mathbf{P}$,

\begin{itemize}
\item The \emph{learning domain} ($\mathrm{LD}$)  is the collection of all copies of the structures from $\mathfrak{K}$. That is,
\[
\mathrm{LD}(\mathfrak{K}):=\bigcup_{\A \in \mathfrak{K}} \{\mathcal{S} : \mathcal{S}\cong \A\}.
\]
As we identify each countable structure with an element of Cantor space, we obtain that $\mathrm{LD}(\mathfrak{K})\subseteq 2^\omega$.
\item The \emph{hypothesis space} ($\mathrm{HS}$) contains, for each $\A\in \mathfrak{K}$, a formal symbol $\ulcorner \A \urcorner$  and a question mark symbol. That is, 
\[
\mathrm{HS}(\mathfrak{K}):=\{\ulcorner \A \urcorner : \A \in  \mathfrak{K} \}\cup \{?\}.
\]
\item A \emph{learner} $\mathbf{M}$ sees, by stages, all positive and negative data about any given structure in the learning domain and is required to output conjectures. This is formalized by saying that $\mathbf{M}$ is a function
\[
\mbox{from $2^{<\omega}$} \mbox{ to } \mathrm{HS}(\mathfrak{K}).
\]
\item The learning is \emph{successful} if, for each structure $\mathcal{S}\in\mathfrak{K}$, the learner eventually
stabilizes to a correct conjecture about its isomorphism type. That is, 
\[
\lim_{n\to \infty} \mathbf{M}(\mathcal{S}\restriction_n)=\ulcorner \mathcal{A}\urcorner  \mbox{ if and only if $\mathcal{S}$ is a copy of $\A$}.
\]
 \end{itemize}
 We say that $\mathfrak{K}$ is \emph{learnable}, if some learner $M$ successfully learns $\mathfrak{K}$.
\end{defn}

\begin{remark}
 In \cite{bazhenov2021turing}, the domain of a learner was limited to 
\[ 
 X:=\{\mathcal{S}\restriction_n \ \colon \mathcal{S}\in \mathrm{LD}(\mathfrak{K}) \},
 \]
  the collection of (finite) initial segments of structures from the learning domain. For our present purposes, it is more convenient to let $\mathbf{M}$ be  defined on all binary strings. This change is not problematic. Indeed, any learner with domain $X$ can be (non-effectively but continuously) transformed into a learner with domain $2^{<\omega}$ by simply   let $h(\sigma)=\ ?$, for all $\sigma\in 2^{<\omega}\smallsetminus X$.
\end{remark}

 In \cite[Theorem~3]{bazhenov2020learning}, we obtained the following model theoretic characterization of which families of structures are learna\-ble.

\begin{thm}[Bazhenov, Fokina, San Mauro] \label{theorem:characterization learning}
Let $\mathfrak{K}:=(\A_i)_{i\in\omega}$ be a countable  family of pairwise nonisomorphic structures. Then,
$\mathfrak{K}$ is learnable  if and only if there are $\Sigma^{\inf}_2$ formulas $\phi_0,\ldots,\phi_n, \ldots$ such that
\[
\A_i \models \phi_j \Leftrightarrow i=j.
\]
\end{thm}

The interested reader is referred to  \cite{bazhenov2020learning} for motivating examples and a detailed discussion about our framework (there named $\Inf\Ex_{\cong}$-learning).

\begin{defn}
We say that a family of structures $\mathfrak{K}$ is \emph{countable} if it contains at most countably many isomorphism types. Similarly,   $\mathfrak{K}$ is \emph{finite}, if it contains only \emph{finitely} many isomorphism types.  In this paper, we won't consider uncountable families.
\end{defn}

\subsection*{Turing computable embeddings} 
We conclude these preliminaries with a brief reminder about the technology of Turing computable embeddings, which was fundamental for proving Theorem~\ref{theorem:characterization learning} and will play a decisive role in Section~\ref{thm:E_3-characterization}.

Turing computable embeddings allow to compare the algorithmic complexity of different isomorphism problems.

\begin{defn} (\cite{CCKM-04,KMV07}).
A Turing operator $\Phi$ is a \emph{Turing computable embedding} of $\mathfrak{K}_0$ into $\mathfrak{K}_1$ (notation: $\mathfrak{K}_0 \leq_{tc} \mathfrak{K}_1$) if it induces an embedding ${\mathfrak{K}_0}_{/\cong}\to {\mathfrak{K}_1}_{/\cong}$, that is,
if $\Phi$ satisﬁes the following:
\begin{itemize}
\item For any $\A \in \mathfrak{K}_0$, the function $\Phi^{\A}$
is the characteristic function of the atomic diagram of a structure from $\mathfrak{K}_1$. This
structure is denoted by $\Phi(\A)$.
\item For any $\A ,\B \in \mathfrak{K}_0$, we have 
\[
\A \cong \B \Leftrightarrow \Phi(\A)\cong \Phi(\B).
\]
\end{itemize}
\end{defn}

It is common to abbreviate the term ``Turing computable embedding'' as \emph{tc-embedding}. One of the most powerful tool in the theory of $tc$-embeddings is the so-called Pullback Theorem~\cite{KMV07}. In this paper, we'll adopt a natural relativization of this result, already employed in \cite{bazhenov2020learning}.  

\begin{thm}[Relativized Pullback Theorem]\label{Pullback Theorem}
Suppose that $X\in\omega$ and $\mathfrak{K}_0 \leq_{tc} \mathfrak{K}_1$ via a Turing $X$-operator $\Phi$. Then, for any $X$-computable
inﬁnitary sentence $\psi$ in the signature of $\mathfrak{K}_1$, one can ﬁnd, effectively with respect to $X$, an $X$-computable inﬁnitary sentence $\psi^\star$ in the
signature of $\mathfrak{K}_0$ such that, for all $\A \in \mathfrak{K}_0$, we have 
\[
\A \models \psi^\star \Leftrightarrow \Phi(\A) \models \psi.
\]
\end{thm}

Note that the Relativized Pullback Theorem can be applied to any continuous operator $\Phi$. Indeed, if $\Phi$ is a continuous, then, by Lemma~\ref{continuous is computable with oracle}, it is equivalent to a Turing $X$-operator for some suitable oracle $X$.

\medskip

We have amassed enough formal ingredients. Let's start.

\section{A new characterization of learnability}
In this section, we offer the promised descriptive set theoretic interpretation of our learning framework. Remember that $E_0$ denotes the relation of eventual agreement of reals, i.e., $\alpha \rel{E_0} \beta$ holds if and only if 
\[
(\exists m)(\forall n\geq m)(\alpha (n)=\beta(n)).\]

\begin{thm}\label{from E_0 to InfEx}
A family of structures $\mathfrak{K}$ is learnable if and only if there is a  continuous function $\Gamma : 2^\omega \to 2^\omega$  such that
\[
\A \rel{\cong} \B \Leftrightarrow \Gamma(\A) \rel{E_0} \, \Gamma(\B),
\]
for all $\A, \B \in \mathrm{LD}(\mathfrak{K})$.
\end{thm}

\begin{proof} For the sake of exposition, we'll assume that $\mathfrak{K}$ is infinite (the other case being easier) and it coincides with $(\A_i)_{i\in\omega}$. Denote $\Gamma(\A_i)$ by $\beta_i$.

\smallskip

 $(\Rightarrow)$: Let $\Gamma$ be a function which induces a continuous reduction from $\mathrm{LD}(\mathfrak{K})_{/ \cong}$ to $E_0$. We need to show that $\mathfrak{K}$ is learnable.  Certainly, $(\beta_i\ \cancel{E_0}\ \beta_j)$, for all $i\neq j$. 
	Since  $\Gamma$ is continuous, by Lemma~\ref{continuous is computable with oracle} there exists an oracle $X\in 2^{\omega}$ and a Turing operator $\Phi$ so that 
	\[
		\Gamma(\alpha) = \Phi^{X\oplus \alpha}, \text{ for every } \alpha\in 2^{\omega}.
	\]
	
	Let $\alpha$ be a real. We define an auxiliary function $f_{sim}(\alpha;i,s)$. Informally speaking, $f_{sim}(\alpha;i,s)$ is a measure of similarity (at the stage $s$) between the reals $\Gamma(\alpha)$ and $\beta_i$.
	
	Let $\ell[s]$ be the greatest number such that for every $x\leq \ell[s]$, the value $\Phi^{(X\oplus \alpha)\upharpoonright s}(x)[s]$ is defined. If there is no such $\ell[s]$, then set $f_{sim}(\alpha;i,s) := -1$ for all $i\in\omega$.
	
	Otherwise, for an index $i\in\omega$, we put
	\[
		f_{sim}(\alpha;i,s) := \begin{cases}
			\max \big\{ k \leq \ell\,\colon \text{for every } j\leq k, & \\
			\phantom{MM} \Phi^{(X\oplus \alpha)\upharpoonright s}(\ell-j) = \beta_i(\ell - j) \big\}, & \text{if } \Phi^{(X\oplus \alpha)\upharpoonright s}(\ell) = \beta_i(\ell);\\
			-1, & \text{otherwise}.
		\end{cases}
	\]
	Here by $\ell$ we denote $\ell[s]$.
	Without loss of generality, we assume that $\ell[s+1] \in \{ \ell[s], \ell[s]+1 \}$.
	
	It is not hard to show that the function $f_{sim}$ satisfies the following properties. Suppose that a real $\alpha$ encodes a copy of the structure $\mathcal{A}_{i_0}$, for some $i_0\in\omega$.
	\begin{itemize}
		\item[(a)] Note that there is an index $m_0$ such that for all $x\geq m_0$, we have $\Gamma(\alpha)(x) = \beta_{i_0}(x)$. This implies that there exists a stage $s_0$ such that every $s\geq s_0$ satisfies $f_{sim}(\alpha;i_0,s+1) \geq f_{sim}(\alpha;i_0,s) > -1$. In addition, 
		\[
			\lim_s f_{sim}(\alpha;i_0,s) = \infty.
		\]
		
		\item[(b)] Let $i\neq i_0$. Since $(\Gamma(\alpha) \ \cancel{E_0}\ \beta_i)$ and $\ell[s+1] \leq \ell[s] +1$ for all $s$, there are infinitely many stages $s$ such that $\Phi^{(X\oplus \alpha)\upharpoonright s}(\ell[s]) \neq \beta_i(\ell[s])$ and $f_{sim}(\alpha;i,s) = -1$. Therefore,
		\[
			\liminf_s f_{sim}(\alpha;i,s) = -1.
		\]
	\end{itemize}
	
	\medskip
	
	\subsection*{Construction} We build our desired learner $\mathbf{M}$. Let $\alpha$ be a real, and let $s$ be a natural number. We put $\mathbf{M}(\text{empty string}) := \ ?$\,.
	
	For the number $s$, we define two auxiliary parameters $b[s] \in \omega$ and $c[s] \in \{ 0,1,2\}$. Informally, these parameters control our learning strategy, i.e., what the learner $\mathbf{M}$ should output. We will ensure the following property: $c[s] = 0$ if and only if $\mathbf{M}(\alpha\upharpoonright s) =\ ?$\,.
	We put $b[0] := 0$ and $c[0]:=0$. If not specified otherwise, we assume that $b[s+1] = b[s]$ and $c[s+1] = c[s]$. 
	
\smallskip	
	
	Suppose that $\mathbf{M}(\alpha\upharpoonright s)$ is defined. We describe how to obtain the value $\mathbf{M}(\alpha \upharpoonright s+1)$. Consider the following three cases:

	\subsubsection*{Case 1} Suppose that $c[s] = 0$. Then we search for the least $j\leq s$ such that $f_{sim}(\alpha;j,s) \neq -1$ and 
	\[ 
		f_{sim}(\alpha; j, s) = \max\{ f_{sim}(\alpha;m,s) \,\colon m\leq s \}.
	\]
	\begin{itemize}
	\item If there is such $j$, then set $\mathbf{M}(\alpha\upharpoonright s+1 ) := \ulcorner \mathcal{A}_j\urcorner$, $c[s+1] := 1$, and $b[s+1]:=s$;
	\item Otherwise, put $\mathbf{M}(\alpha\upharpoonright s+1) :=\ ?$\,.
\end{itemize}

	\subsubsection*{Case 2} Suppose that $c[s] = 1$ and $\mathbf{M}(\alpha \upharpoonright s) = \ulcorner \mathcal{A}_i\urcorner$. 
	
	\begin{itemize}
	\item If $f_{sim}(\alpha;i,s+1) \neq -1$, then define $\mathbf{M}(\alpha\upharpoonright s+1 ) := \ulcorner \mathcal{A}_i\urcorner$;
	\item If $f_{sim}(\alpha;i,s+1) = -1$, then put $\mathbf{M}(\alpha\upharpoonright s+1 ) := \ulcorner \mathcal{A}_{0}\urcorner$ and $c[s]:=2$. 
	\end{itemize}

	\subsubsection*{Case 3}  Suppose that $c[s] = 2$ and $\mathbf{M}(\alpha \upharpoonright s) = \ulcorner \mathcal{A}_i\urcorner$.  Our construction will ensure that in this case, we have $i \leq b[s]$.
	\begin{itemize}
	\item If $f_{sim}(\alpha;i,s+1) \neq -1$, then $\mathbf{M}(\alpha\upharpoonright s+1 ) := \ulcorner \mathcal{A}_i\urcorner$;
	\item If $f_{sim}(\alpha;i,s+1) = -1$ and $i < b[s]$, then set $\mathbf{M}(\alpha\upharpoonright s+1 ) := \ulcorner \mathcal{A}_{i+1}\urcorner$;
	\item 	If $f_{sim}(\alpha;i,s+1) = -1$ and $i = b[s]$, then define $\mathbf{M}(\alpha\upharpoonright s+1 ) := \ ?$ and $c[s]:=0$.
	\end{itemize}

	\subsection*{Verification} We show that $\mathbf{M}$ learns our family $\mathfrak{K}$. Suppose that $\alpha$ is a real, which encodes a copy $\mathcal{S}$ of some $\mathcal{A}_{i_0}$.
		
	Let $s_0$ be a stage such that $f_{sim}(\alpha;i_0,s) \neq -1$ for all $s \geq s_0$. \emph{Case~1} of the construction ensures that there are infinitely many stages $s \geq s_0$ such that $\mathbf{M}(\alpha \upharpoonright s) = \ulcorner \mathcal{A}_j \urcorner$ for some $j\in\omega$. 
	
	If there exists a stage $s_1 \geq s_0$ such that $\mathbf{M}(\alpha \upharpoonright s_1) = \ulcorner \mathcal{A}_{i_0} \urcorner$, then we have
	\[
		\mathbf{M}(\alpha \upharpoonright s) = \ulcorner \mathcal{A}_{i_0} \urcorner \text{ for all } s \geq s_1,
	\]
	and in the limit, $\mathbf{M}$ outputs the correct conjecture. Thus, it is sufficient to establish the existence of this $s_1$.
	
	Let $s'$ be an arbitrary stage such that $s'\geq s_0$ and $\mathbf{M}(\alpha \upharpoonright s') = \ulcorner \mathcal{A}_j \urcorner$ for some $j$. If $j = i_0$, then there is nothing to prove. 
	Assume that $j\neq i_0$. There are again three cases:
	
	\subsubsection*{Case I} If $c[s'] = 1$, then eventually we will witness a sequence of stages $s' \leq s'' < s''_0 < s''_1 <\dots < s''_{i_0-1}$ such that
	\begin{multline*}
		f_{sim}(\alpha;j,s'') = f_{sim}(\alpha; 0, s''_0) = f_{sim}(\alpha; 1, s''_1) = \dots =\\ 
		= f_{sim}(\alpha; i_0-1, s''_{i_0-1}) = -1.
	\end{multline*}
	Hence, the construction ensures that $\mathbf{M}(\alpha \upharpoonright s''_{i_0-1}) = \ulcorner \mathcal{A}_{i_0} \urcorner$.
	
	\subsubsection*{Case II} If $c[s'] = 2$ and $j<i_0$, then we will find a sequence of stages $s' \leq s''_j < s''_{j+1} <\dots < s''_{i_0-1}$ such that
	\[
		f_{sim}(\alpha; k, s''_k) = -1 \text{ for every } j\leq k <i_0.
	\]
	Again, we have $\mathbf{M}(\alpha \upharpoonright s''_{i_0-1}) = \ulcorner \mathcal{A}_{i_0} \urcorner$.
	
	\subsubsection*{Case III} If $c[s'] = 2$ and $j > i_0$, then there will be a sequence of stages $s' \leq s''_j < s''_{j+1} < \dots < s''_{b[s']}$ such that
	\[
		f_{sim}(\alpha; k, s''_k) = -1 \text{ for every } j\leq k \leq b[s'].
	\]
	Then for $s^{\ast} := s''_{b[s']} + 1$, we have $c[s^{\ast}] = 1$ and $\mathbf{M}(\alpha \upharpoonright s^{\ast}) = \ulcorner \mathcal{A}_{\ell} \urcorner$ for some $\ell$. After that, the argument proceeds similarly to the first case.
	
\smallskip	
	
	Therefore, we proved that there is $s_1 \geq s_0$ with $\mathbf{M}(\alpha \upharpoonright s_1) = \ulcorner \mathcal{A}_{i_0} \urcorner$. This implies that in the limit, $\mathbf{M}$ says that ``$\mathcal{S}$ is a copy of $\mathcal{A}_{i_0}$'', and the family $\mathfrak{K}$ is learnable by $\mathbf{M}$.

\medskip

$(\Leftarrow)$: For the converse direction, let $\mathbf{M}$ be a learner of $\mathfrak{K}$. We need to construct a continuous  function $\Gamma : 2^\omega \to 2^\omega$ which induces a reduction from $\mathrm{LD}(\mathfrak{K})_{/ \cong}$ to $E_0$. To this end, it suffices to fix a countably infinite transversal $(\alpha_i)_{i\in\omega}$ of $E_0$ (i.e. a set intersecting countably many equivalence classes of $E_0$ in exactly one point) and define $\Gamma$ as follows,
\[
\Gamma(\beta)(s):=\alpha_{\mathbf{M}(\beta\restriction_s)}{(s)}.
\] 
Here we use the following convention:
\begin{itemize}
    \item if $\mathbf{M}(\beta\restriction_s) = \ulcorner \mathcal{A}_i\urcorner$, then $\alpha_{\mathbf{M}(\beta\restriction_s)} = \alpha_i$;
    
    \item if $\mathbf{M}(\beta\restriction_s) = \ ?$, then $\alpha_{\mathbf{M}(\beta\restriction_s)} = 0^{\infty}$.
\end{itemize}

To verify that this $\Gamma$ works, it is enough to observe the following: if a real $\beta$ encodes a copy of some $\A_i$ from $\mathrm{LD}(\mathfrak{K})$, then   there must be a stage $s_0$ such that, for all $s\geq s_0$, $\mathbf{M}(\beta \restriction_s)$ outputs $\ulcorner \mathcal{A}_i \urcorner$, and thus $\Gamma(\beta)$ is $E_0$-equivalent to $\alpha_i$.
So, since the $\alpha_i$'s form a transversal for $E_0$, we deduce that, if $\beta_0$ and $\beta_1$ encode copies of $\A_i$ and $\A_j$ respectively,  then
\[
\Gamma(\beta_0)\rel{E_0}\Gamma(\beta_1) \Leftrightarrow i=j.
\]
This concludes the proof.
\end{proof}

The above theorem unlocks a natural way to stratify learning
problems, by simply replacing $E_0$ with Borel equivalence
relations of higher and higher complexity.

\begin{defn}
A family of structures $\mathfrak{K}$ is \emph{$E$-learnable} if there is
 function $\Gamma : 2^\omega\to 2^\omega$ which continuously
reduce $\mathrm{LD}(\mathfrak{K})_{/\cong}$ to $E$.
\end{defn}

\begin{defn}
Let $E,F$ be Borel equivalence relations. $E$ is \emph{$\mathrm{Learn}^\omega$-reducible} to $F$, if every countable $E$-learnable
family is also $F$-learnable. $E$ is \emph{$\mathrm{Learn}^{<\omega}$-reducible} to $F$, if
every finite $E$-learnable family is also $F$-learnable.
\end{defn}

The rest of the paper is devoted to the study of the learning of power of some benchmark Borel equivalence relations.


\section{When oracle equivalence relations don't help}

In this section, we analyze the learning power of $E_1$ and $E_2$. These equivalence relations are incomparable and strictly above $E_0$ with respect to continuous reductions; in fact, the same is true if one requires computable reductions. But, as is proven in Theorems~\ref{prop:E_01} and \ref{thm:E_02}, $E_1$ and $E_2$ coincide and collapse to $E_0$ with respect to their learning power.


\subsection{$E_1$-learning} Recall that the equivalence relation $E_1$ is given by
	\[
(\alpha\rel{E_1}\beta) \Leftrightarrow		(\forall^{\infty} m\in\omega) (\alpha^{[m]} = \beta^{[m]}).
	\]

\begin{thm}\label{prop:E_01}
	A family $\mathfrak{K}$ is $E_1$-learnable if and only if $\mathfrak{K}$ is $E_0$-learnable. That is, $E_1$ and $E_0$ are $\mathrm{Learn}^\omega$-equivalent.
\end{thm}
\begin{proof}
 	Since $E_0$ is continuously reducible to $E_1$ (see Figure~\ref{fig:benchmark}), every $E_0$-learnable family is also $E_1$-le\-arn\-able.
 	
	On the other hand, let $\mathfrak{K} := (\mathcal{A}_i)_{i\in\omega}$ be an $E_1$-learnable family. Let $\Gamma: 2^\omega \to 2^\omega$ induce a continuous reduction from $\mathrm{LD}(\mathfrak{K})_{/\cong}$ to $E_1$. For each $i\in \omega$, we choose a real $\beta_i$ such that $\Gamma$ maps all copies of $\mathcal{A}_i$ into the class $[\beta_i]_{E_1}$.
	Fix a computable bijection $\xi$ from the set $\{ (i,j) \in\omega^2 \,\colon i\neq j\} \times \omega$ onto $\omega$.
	
	We build a set $X = \{ m_s\,\colon s\in\omega\}$ as follows. Put $m_0 := 0$. 
		Suppose that $s = \xi(i,j,t)$ and $m_s$ is already defined. Since $(\beta_i \ \cancel{E_1}\ \beta_j)$, there exists the least $q > m_s$ such that $\beta^{[q]}_i \neq \beta^{[q]}_j$. We choose the least $\ell \in\omega$ with $\beta^{[q]}_i (\ell) \neq \beta^{[q]}_j(\ell)$, and put $m_{s+1} := \langle q,\ell \rangle$.
	It is not hard to see that for every $q\in\omega$, there is at most one $\ell$ such that $\langle q, \ell\rangle$ belongs to $X$. 
	
	We define an operator $\Psi$ as follows: for every real $\alpha$ and $s \in\omega$, set
	\[
		\Psi(\alpha)(s) := \alpha(m_s).
	\]
	It is clear that the operator $\Psi$ is $X$-computable~--- hence, $\Psi$ is continuous. 
	
	We show that the operator $\Phi := \Psi\circ \Gamma$ provides a continuous reduction from $\mathrm{LD}(\mathfrak{K})_{/\cong}$ to $E_0$. Let $\alpha$ be a real which encodes a copy of some $\mathcal{A}_{i_0}$. Since $(\Gamma(\alpha)\ E_1\ \beta_{i_0})$, almost every $s\in\omega$ satisfies $\Gamma(\alpha)(m_s) = \beta_{i_0}(m_s)$. Thus, $\Phi(\alpha)$ is $E_0$-equivalent to $\Psi(\beta_{i_0})$.
	
	Suppose that $i\neq i_0$. Then for almost all $t\in\omega$, we have
	\[
		\Gamma(\alpha)( m_{\xi(i,i_0,t)} ) = \beta_{i_0}( m_{\xi(i,i_0,t)} ) \neq \beta_{i}( m_{\xi(i,i_0,t)} ).
	\]
	This implies that $(\Phi(\alpha) \ \cancel{E_0}\ \Psi(\beta_i))$. Therefore, we deduce that our family $\mathfrak{K}$	is $E_0$-learnable.	
	The theorem is proved.
\end{proof}

\subsection{$E_2$-learning} Recall that the equivalence relation $E_2$ is given by
\[
\alpha \ E_2\ \beta \Leftrightarrow
		\sum_{k=0}^{\infty} \frac{(\alpha\triangle\beta)(k)}{k+1}\ <\ \infty.
	\]

\begin{thm}\label{thm:E_02}
	A countable family $\mathfrak{K}$ is $E_2$-learnable if and only if $\mathfrak{K}$ is $E_0$-learnable. That is, $E_2$ and $E_0$ are $\mathrm{Learn}^\omega$-equivalent.
\end{thm}
\begin{proof}
 	Since $E_0$ is continuously reducible to $E_2$ (see Figure~\ref{fig:benchmark}), every $E_0$-learnable family is $E_2$-le\-arn\-able.

	Let $\mathfrak{K} := (\mathcal{A}_i)_{i\in\omega}$ be an $E_2$-learnable family. Let $\Gamma$ be an operator, which induces a continuous reduction from $LD(\mathfrak{K})_{/\cong}$ to $E_2$. For $i\in \omega$, we fix a real $\beta_i$ such that $\Gamma$ maps all copies of $\mathcal{A}_i$ into $[\beta_i]_{E_2}$.
	
	By Lemma \ref{continuous is computable with oracle}, there exist an oracle $X$ and a Turing operator $\Phi$ such that $\Gamma(\alpha) = \Phi^{X\oplus \alpha}$ for all $\alpha \in 2^{\omega}$.
	Consider the uniform join
	\[
		Y := X\oplus \bigoplus_{i\in\omega} \beta_i.
	\]
	
	\subsection*{Construction}
	We define a $Y$-computable operator $\Psi$.	
	For a real $\alpha$, we describe how to construct the real $\gamma_{\alpha} = \Psi(\alpha)$.
	For $s\in\omega$, by $\ell[s]$ we denote the greatest number such that for every $x\leq \ell[s]$, the value $\Phi^{(X\oplus \alpha)\upharpoonright s}(x)[s]$ is defined. Without loss of generality, one may assume that $\ell[s]$ is defined for every $s$.
	
	For $i, s\in\omega$, we consider the partial sum
	\[
		p(i,s) := \sum^{\ell[s]}_{ k=0} \frac{(\beta_i \triangle \Phi^{X\oplus \alpha})(k)}{k+1}.
	\]
	
	At a stage $s$, we define auxiliary values $i[s], b[s]\in \omega$ and $c[s] \in \{ 0,1\}$. Similarly to the proof of Theorem~\ref{from E_0 to InfEx}, these parameters control the flow of the construction. 
	Moreover, at each stage $s$, we set $\gamma_{\alpha}(s) := \beta_{i[s]}(s)$. Our construction will ensure that $i[s] \leq b[s]$ for every $s$.
	
	\subsubsection*{Stage 0} Set $i[0] = 0$, $b[0] = 1$, and $c[0] = 0$.
	
	\subsubsection*{Stage s+1} We assume that the parameters $b[s]$, $c[s]$, and $i[s]$ are already defined. Consider the following four cases:
	
	\subsubsection*{Case~1.} If $p(i[s], s+1) \leq b[s]$, then do not change anything.
	
	\subsubsection*{Case~2.} If $p(i[s],s+1) > b[s]$ and $c[s] = 0$, then put $i[s+1] := 0$ and $c[s] := 1$.
	\subsubsection*{Case~3.} Suppose that $p(i[s],s+1) > b[s]$, $c[s] = 1$, and $i[s] < b[s]$. Define $i[s+1] := i[s] + 1$.
	
	\subsubsection*{Case~4.} Suppose that $p(i[s],s+1) > b[s]$, $c[s] = 1$, and $i[s] = b[s]$. Find the least $i_0\leq b[s]+1$ such that 
	\[
		p(i_0, s+1) = \min\{ p(j,s+1)\,\colon j \leq b[s] +1 \}.
	\]
	We put $i[s+1] := i_0$, $c[s+1] := 0$, and 
	\[
		b[s + 1] := \max (b[s] + 1, \text{the integer part of } p(i_0,s+1) + 1).
	\]
	
	This concludes the description of the construction. It is clear that the operator $\Psi\colon \alpha \mapsto \gamma_{\alpha}$ is $Y$-computable.

\subsection*{Verification}	
	Suppose that a real $\alpha$ encodes a copy of the structure $\mathcal{A}_{i_0}$. We define:
	\[
		N_0 := \sum^{\infty}_{k=0} \frac{(\beta_{i_0} \triangle \Gamma(\alpha))(k)}{k+1}.
	\]
	
	\begin{claim}
		There exists a finite limit $b^{\ast} = \lim_s b[s]$. In addition, $b^{\ast} \geq i_0$.
	\end{claim}
	\begin{proof}
		We distinguish two cases. First, assume that $b[s] < i_0$ for all $s$. Then we have $i[s] < i_0$ for every $s$. Furthermore, since the sequence $b[s]$ is non-decreasing, there exists $b^{\ast} = \lim_s b[s]$ with $b^{\ast} < i_0$.
		
		Since $(\Gamma(\alpha)\ \cancel{E_2}\ \beta_j)$ for all $j\neq i_0$, there exists a stage $s_0$ such that $p(j,s_0) > i_0$ for all $j<i_0$, and $b[s] = b^{\ast}$ for all $s \geq s^{\ast}$. Then, our construction ensures that after the stage $s_0$, there will be a stage $s_1$ satisfying \emph{Case~4}. This implies that $b[s_1] \geq b^{\ast} +1$, which gives a contradiction. Thus, we deduce that there must exist a stage $s'_0$ such that $b[s'_0] \geq i_0$.
		
		Second, assume that $\lim_s b[s] = \infty$. This implies that there are infinitely many stages $s > s'_0$ satisfying \emph{Case~4}. Choose a stage $s_1 > s'_0$ such that $s_1$ satisfies \emph{Case~4} and $b[s_1] \geq N_0 + 1$. Consider the value $i^{\ast} := i[s_1]$.
		
\begin{itemize}
\item		If $i^{\ast} = i_0$, then for every $s$, we have $p(i^{\ast}, s) < b[s_1]$. This implies that every stage $s > s_1$ satisfies \emph{Case~1}, which gives a contradiction.
\item		If $i^{\ast} \neq i_0$, then find the least stage $s_2 > s_1$ with $p(i^{\ast}, s_2) > b[s_1]$. Then the stage $s_2$ satisfies \emph{Case~2}, and we have $c[s_2] = 1$. Therefore, \emph{Case~3} of the construction ensures that there is a sequence of stages 
\[		
		s_2 = s''_0 < s''_1 < \dots < s''_{i_0}
\]		
		 such that $i[s''_k] = k$ for every $k\leq i_0$. Again, every stage $s>s''_{i_0}$ satisfies \emph{Case~1}, which provides a contradiction.
\end{itemize}	
		Therefore, we proved that there is a finite limit $b^{\ast} = \lim_s b[s]$, and $b^{\ast} \geq i_0$.
	\end{proof}
	
	Now choose a stage $s^{\ast}$ such that $b[s^{\ast}] = b^{\ast}$. There exists a stage $s_1 \geq s^{\ast}$ such that every $i \leq b^{\ast}$ satisfies the following: if $i\neq i_0$, then $p(i,s_1) > b^{\ast}$. Since after the stage $s^{\ast}$, there are no stages satisfying \emph{Case~4}, it is not hard to deduce that for every $s \geq s_1 + b^{\ast} + 2$, we must have $i[s] = i_0$. 
	
	This implies that the real $\Psi(\alpha)$ is $E_0$-equivalent to $\beta_{i_0}$. For all $i\neq j$, we have $(\beta_i \ \cancel{E_2}\ \beta_j)$~--- clearly, this implies $(\beta_i \ \cancel{E_0}\ \beta_j)$. Hence, we conclude that our operator $\Psi$ provides a continuous reduction from $\mathrm{LD}(\mathfrak{K})_{/\cong}$ to $E_0$. In other words, the family $\mathfrak{K}$ is $E_0$-learnable, as desired.
\end{proof}


\section{Characterizing the learning power of $E_3$}
All equivalence relations considered so far (i.e., $E_0$, $E_1$, and $E_2$) are inseparable with respect to their learning power. In fact, by Theorem \ref{from E_0 to InfEx}, they don't expand the boundaries of our original framework. The case of $E_3$, to be discussed in this section, is different. Namely, $E_3$ has strictly more learning power than $E_0$---but this fact is only witnessed by infinite families. Recall that the equivalence relation $E_3$ is given by
	\[
(\alpha\rel{E_3}\beta) \Leftrightarrow		(\forall m\in\omega) (\alpha^{[m]} \rel{E_0} \beta^{[m]}).
	\]
   
	\begin{thm}\label{prop:E_03}
		A finite family $\mathfrak{K}$ is $E_3$-learnable if and only if $\mathfrak{K}$ is $E_0$-learnable. That is, $E_3$ and $E_0$ are $\mathrm{Learn}^{<\omega}$-equivalent.
	\end{thm}
	\begin{proof}
	One direction is again immediate: since $E_0$ is continuously reducible to $E_3$ (see Figure~\ref{fig:benchmark}), every $E_0$-learnable family is $E_3$-learnable.
	
	For the other direction, let $\mathfrak{K} := (\A_i)_{i\in\omega}$ be an $E_3$-learnable family and let $\Gamma$ induce a continuous reduction from $LD(\mathfrak{K})_{/\cong}$ to $E_3$. For  $i\leq n$, choose $\beta_i$ such that $\Gamma$ maps all copies of $\mathcal{A}_i$ into $[\beta_i]_{E_3}$.
	For each pair of indices $i\neq j$, we choose a number $q(i,j)$ such that 
	\[
		\beta_{i}^{[q(i,j)]}\ \cancel{E_0}\ \beta_j^{[q(i,j)]}.
	\]	
	
	Then, we define a Turing operator $\Psi \colon 2^{\omega} \to 2^{\omega}$ as follows.
	\[
		\Psi(\alpha) = \bigoplus_{i\neq j \leq n} \alpha^{[q(i,j)]}.
	\]
	The operator $\Phi := \Psi\circ \Gamma$ provides a continuous reduction from $LD(\mathfrak{K})_{/\cong}$ to $E_0$. Indeed, let $\alpha$ be a real which encodes a copy of $\mathcal{A}_{i_0}$. Then $(\Gamma(\alpha)\ E_3\ \beta_{i_0})$ and $(\Phi(\alpha) \ E_0\ \Psi(\beta_{i_0}))$. If $i\neq i_0$, then we have 
	\[
		(\alpha^{[q(i,i_0)]}\ E_0\ \beta_{i_0}^{[q(i,i_0)]}\ \cancel{E_0}\ \beta_i^{[q(i,i_0)]}) \text{ and } (\Phi(\alpha) \ \cancel{E_0}\ \Psi(\beta_{i})).
	\]
	Therefore, the family $\mathfrak{K}$ is $E_0$-learnable.
\end{proof}

Our next result separates $E_3$-learnability and $E_0$-learnability, thus proving that $E_0$ is strictly $\mathrm{Learn}^\omega$-reducible to $E_3$

\begin{thm}\label{thm:E_03_for_countable}
	There exists an infinite  family $\mathfrak{K} := (\mathcal{A}_i)_{i\in\omega}$ which is $E_3$-learnable, but not $E_0$-learnable.
\end{thm}
\begin{proof}
	For the sake of exposition, first we give proof for the case, when the signature of the class $\mathfrak{K}$ is allowed to be infinite. After that, we provide comments on how to build the desired $\mathfrak{K}$ as a family of directed graphs.
	
	Consider signature $L = \{ R_j\,\colon j\in\omega \} \cup \{ \leq\}$, where $R_j$ are unary predicates. Given a real $\alpha$, we define an $L$-structure $\mathcal{D}(\alpha)$ as follows:
	\begin{itemize}
		\item Inside $\mathcal{D}(\alpha)$, the relations $R_j$, $j\in\omega$, are pairwise disjoint. We say that the set $R_j^{\mathcal{D}(\alpha)}$ is the \emph{$R_j$-box} of $\mathcal{D}(\alpha)$.
	 
	 	\item The $R_j$-box of $\mathcal{D}(\alpha)$ contains a linear order $L_j$ such that
	 	\[
	 		L_j \cong \begin{cases}
	 			\omega, & \text{if } \alpha(j) = 0,\\
	 			\omega^{\ast}, & \text{if } \alpha(j) = 1.
	 		\end{cases}
	 	\]
	\end{itemize}
	where $\omega$ and $\omega^*$ are respectively the order types of the positive and negative integers.
	For a finite string $\sigma\in 2^{<\omega}$, let $\mathcal{A}_{\sigma}$ be the structure $\mathcal{D}(\sigma \widehat{\ } 10^{\infty})$. Our family $\mathfrak{K}$ consists of all $\mathcal{A}_{\sigma}$, $\sigma\in 2^{<\omega}$.
	
	\begin{lemma}\label{lem:K_is_E3-learnable}
		The family $\mathfrak{K}$ is $E_3$-learnable.
	\end{lemma}
	\begin{proof}
		Recall that the family $\{ \omega, \omega^{\ast}\}$ is learnable, as they are  distinguishable by $\Sigma_2^{\inf}$ formulas~\cite[Theorem~3]{bazhenov2020learning}. By employing this fact, it is not hard to build a Turing operator $\Phi$, which acts as follows. Given a real $\alpha$, it treats $\alpha$ as a code for the atomic diagram of a countable partial order $\mathcal{L}$. Then:
		\begin{itemize}
			\item If $\mathcal{L}$ is a copy of $\omega$, then the output $\Phi(\alpha)$ is $E_0$-equivalent to $0^{\infty}$.
			
			\item If $\mathcal{L} \cong \omega^{\ast}$, then we have $(\Phi(\alpha)\ E_0\ 1^{\infty})$.
		\end{itemize}
	
	For each index $j\in\omega$, we define a Turing operator $\Psi_j$. Given a real $\alpha$, it treats $\alpha$ as a code of a countable $L$-structure $\mathcal{A}$. The output $\Psi_j(\alpha)$ encodes the partial order, which is contained inside the $R_j$-box of $\mathcal{A}$.
	
	Finally, we define an operator $\Theta$. For $\alpha\in 2^{\omega}$ and for $j,k\in\omega$, we set
	\[
		\Theta(\alpha)(\langle j,k\rangle) := (\Phi\circ \Psi_j(\alpha))(k).
	\]
	
	Observe the following. Let $\beta$ be a real. If a real $\alpha$ encodes a copy of the structure $\mathcal{D}(\beta)$, then for every $j\in\omega$, we have:
	\begin{itemize}
		\item if $\beta(j) = 0$, then the $j$-th column $(\Theta(\alpha))^{[j]}$ is $E_0$-equivalent to $0^{\infty}$;
		
		\item if $\beta(j) = 1$, then $(\Theta(\alpha))^{[j]}\ E_0 \ 1^{\infty}$.
	\end{itemize}
	This observation implies that the operator $\Theta$ witnesses the $E_3$-learna\-bi\-li\-ty of our family $\mathfrak{K}$. Lemma~\ref{lem:K_is_E3-learnable} is proved.
	\end{proof}
	
	Now, towards a contradiction, assume that the family $\mathfrak{K}$ is $E_0$-learnable. Then $\mathfrak{K}$ is \textbf{InfEx}-learnable, and by Theorem~\ref{theorem:characterization learning}, 
	 one can choose an infinitary $\Sigma_2$ sentence $\theta$ such that $\mathcal{A}_0 \models \theta$ and for every $\sigma \neq 0$, we have $\mathcal{A}_{\sigma} \not\models \theta$.
	
	Without loss of generality, one may assume that
	\[
		\theta = \exists \bar x\  \underset{i\in I}{\bigwedge\skipmm{6.8}\bigwedge}\ \forall \bar y_i \psi_i (\bar x, \bar y_i),
	\]
	where every $\psi_i$ is a quantifier-free formula. Fix a tuple $\bar c$ from the structure $\mathcal{A}_0$ such that
	\[
		\mathcal{A}_0 \models \underset{i\in I}{\bigwedge\skipmm{6.8}\bigwedge}\  \forall \bar y_i \psi_i (\bar c, \bar y_i).
	\]
	Choose a natural number $N$ such that for every $j\geq N$, the $R_j$-box of $\mathcal{A}_0$ does not contain elements from $\bar c$.
	
	Consider a string $\tau := 010^N$ and the corresponding structure $\mathcal{A}_{\tau} = \mathcal{D}(\tau \widehat{\ } 10^{\infty})$. It is clear that for every $j < N$, the (contents of the) $R_j$-boxes inside $\mathcal{A}_0$ and $\mathcal{A}_{\tau}$ are isomorphic. Therefore, one can choose a tuple $\bar d$ inside $\mathcal{A}_{\tau}$ as isomorphic copies of $\bar c$ (with respect to the isomorphism of the $R_j$-boxes, $j < N$).
	
	\begin{claim}\label{claim:same_E_sentences}
		The structures $(\mathcal{A}_0, \bar c)$ and $(\mathcal{A}_{\tau}, \bar d)$ satisfy the same $\exists$-sentences.
	\end{claim}
	\begin{proof}
		It is sufficient to establish the following. Every quantifier-free formula $\psi(\bar x, \bar y)$ satisfies 
		\[
			\mathcal{A}_0 \models \exists \bar y \psi(\bar c, \bar y)\  \ \Rightarrow\ \ \mathcal{A}_{\tau} \models \exists \bar y \psi(\bar d, \bar y).
		\]
		The other direction ($\Leftarrow$) can be obtained via a similar argument.
		
		Choose a tuple $\bar b$ from $\mathcal{A}_0$ such that $\mathcal{A}_0 \models \psi(\bar c, \bar b)$. Suppose that $\bar b = b_0, b_1, \dots, b_m$. 
		We define a new tuple $\bar b' = b'_0,b'_1,\dots,b'_m$ from $\mathcal{A}_{\tau}$ as follows:
		\begin{itemize}
			\item If $b_k$ lies in an $R_j$-box, which contains elements from $\bar c$, then $b'_k$ is defined as the copy of $b_k$ with respect to the natural isomorphism of $R_j$-boxes, $j < N$.
			
			\item Suppose that $b_k$ belongs to an $R_j$-box, which does not contain elements from $\bar c$. Then $b'_k$ can be chosen as any element from the $R_j$-box of $\mathcal{A}_{\tau}$, while preserving the ordering $\leq$. More formally, one needs to ensure the following: if $b_k \neq b_{\ell}$ both belong to this $R_j$-box, then we have:
			\[
				\mathcal{A}_0 \models b_k \leq b_{\ell}\ \Leftrightarrow\ \mathcal{A}_{\tau} \models b'_k \leq b'_{\ell}.
			\]
		\end{itemize}
		
		It is clear that the tuples $\bar c,\bar b$ and $\bar d,\bar b'$ satisfy the same atomic formulas. Therefore, we deduce that the structure $\mathcal{A}_{\tau}$ satisfies $\psi(\bar d, \bar b')$, and $\mathcal{A}_{\tau} \models \exists \bar y \psi(\bar d, \bar y)$.
	\end{proof}
	
	Claim~\ref{claim:same_E_sentences} implies that
	\[
		\mathcal{A}_{\tau} \models \underset{i\in I}{\bigwedge\skipmm{6.8}\bigwedge}\  \forall \bar y_i \psi_i (\bar d, \bar y_i),
	\]
	and hence, $\mathcal{A}_{\tau} \models \theta$, which contradicts the choice of $\theta$. We deduce that the family $\mathfrak{K}$ is not $E_0$-learnable. 
	
	\smallskip
	
	In order to obtain a family of directed graphs $\mathfrak{K}_{gr}$, which has the same properties as the family $\mathfrak{K}$, one can proceed as follows. Instead of distinguishing an $R_j$-box via the predicate $R_j$, one attaches to each element $a$ of the (intended) $R_j$-box its own cycle of size $(j+3)$. After that, the proof for the family $\mathfrak{K}_{gr}$ is essentially the same as the one provided above. Theorem~\ref{thm:E_03_for_countable} is proved.
\end{proof}


\subsection{A syntactic characterization of $E_3$-learnability}\label{section: characterization of E_3} As aforementioned, in the previous work we obtained a full syntactic characterization of which families of structures are learnable, by means of $\Sigma^{\mathrm{inf}}_2$ formulas (see Theorem~\ref{theorem:characterization learning}). The next theorem offers an analogous characterization for $E_3$-learning.

\begin{thm}\label{thm:E_3-characterization}
	Let $\mathfrak{K} := (\mathcal{A}_i)_{ i\in\omega}$ be a countable family. The family $\mathfrak{K}$ is $E_3$-learnable if and only if there exists a countable family of $\Sigma^{\mathrm{inf}}_2$ sentences $\Theta$ with the following properties:
	\begin{itemize}
		\item[(a)] if $\theta$ is a formula from $\Theta$, then there is a formula $\psi \in \Theta$ such that for every $\mathcal{A} \in \mathfrak{K}$,
		\[
			\mathcal{A} \models \theta\ \Leftrightarrow\ \mathcal{A} \models \neg \psi;
		\]
		
		\item[(b)] if $\mathcal{A}\not\cong \mathcal{B}$ are structures from $\mathfrak{K}$, then there is a sentence $\theta \in \Theta$ such that
		\[
			\mathcal{A} \models \theta \text{ and } \mathcal{B} \models \neg\theta.
		\]
	\end{itemize}
\end{thm}
\begin{proof}
The proof of the theorem is inspired by ideas from  \cite{bazhenov2020learning}. In particular, we will adopt the technology of $tc$-embeddings and the Relativized Pullback Theorem reminded in the preliminaries. 	
	\smallskip

	Consider a signature $L_{st} = \{ \leq\} \cup \{ P_i \,\colon i\in\omega\}$, where $P_i$ are unary predicates. For an index $i\in\omega$, an $L$-structure $\mathcal{S}_i$ satisfies the following properties:
\begin{itemize}
\item  Inside $\mathcal{S}_i$, the relations $P_j$ are pairwise disjoint. In addition, if $x\in P_j$ and $y\in P_{k}$ for some $j\neq k$, then $x$ and $y$ are $\leq$-incomparable. Let $\eta$ be the order type of the rational numbers.
\item The predicate $P_i$ contains an isomorphic copy of $1+\eta$.
\item Every $P_j$, for $j\neq i$, contains a copy of $\eta$.
\end{itemize}	
	   
	The class $\mathfrak{K}_{st}$ consists of all structures $\mathcal{S}_i$ with $i\in\omega$.
	
	 In \cite{bazhenov2020learning}, it is  shown $\mathfrak{K}_{st}$ is an archetypical $E_0$-learnable  family, in the sense that  a  countable family $\mathfrak{C}$ is learnable if and only if there is a continuous embedding from the class $\mathfrak{C}$ into $\mathfrak{K}_{st}$.

	\smallskip
	
	For dealing with $E_3$-learnability, we have to introduce a new, and more complicated, class $\mathfrak{C}_{st}$. But the informal idea behind $\mathfrak{C}_{st}$ is pretty simple: roughly speaking, this class contains all countable disjoint sums of the structures from $\mathfrak{K}_{st}$.
	
	Consider a new signature $L_1 = L_{st} \cup \{ Q_k\,\colon k\in\omega\}$, where $Q_k$ are unary predicates. The class $\mathfrak{C}_{st}$ contains all $L$-structures $\mathcal{M}$, which satisfy the following properties:
	\begin{itemize}
		\item Their relations $Q_k$, $k\in\omega$, are pairwise disjoint. We say that the $L_{st}$-substructure with domain $\mathcal{M}\upharpoonright Q_k$ is the \emph{$Q_k$-box} of $\mathcal{M}$.
		
		\item Every $Q_k$-box of $\mathcal{M}$ is isomorphic to a structure from the class $\mathfrak{K}_{st}$.
	\end{itemize}
	Note that our class $\mathfrak{C}_{st}$ has cardinality $2^{\aleph_0}$.
	
	\begin{lemma}\label{lem:C_st_has_a_good_family}
		The class $\mathfrak{C}_{st}$ has a computable family of $\Sigma^{\mathrm{inf}}_2$ sentences $\Theta$, which satisfies properties~(a) and~(b) from the formulation of Theorem~\ref{thm:E_3-characterization}.
	\end{lemma}
	\begin{proof}
		The desired family $\Theta$ contains the following $\Sigma^{\mathrm{inf}}_2$ sentences:
		\begin{enumerate}
			\item For each $i$ and $j$, we add a finitary $\Sigma_2$ sentence $\theta_{i,j}$, which states the following: ``the $P_j$-predicate inside the $Q_i$-box has a $\leq$-least element''.
			
			\item For each $i$ and $j$, we add a $\Sigma^{\mathrm{inf}}_2$ sentence $\psi_{i,j}$, which is equivalent to the following formula:
			\[
				\underset{k\neq j}{\bigvee\skipmm{6.8}\bigvee}\  \theta_{i,k}.
			\]
			In other words, there is some $k\neq j$ such that the $P_k$-predicate inside the $Q_i$-box possesses the least element.
		\end{enumerate}
		
		Let $\mathcal{M}$ be an arbitrary structure from $\mathfrak{C}_{st}$. Since the $Q_i$-box of $\mathcal{M}$ is a structure from $\mathfrak{K}_{st}$, it is not hard to show that
		\[
			\mathcal{M} \models \theta_{i,j}\ \Leftrightarrow\ \mathcal{M} \models \neg \psi_{i,j}.
		\]
		Hence, we deduce that the class $\mathfrak{C}_{st}$ satisfies property~(a) of Theorem~\ref{thm:E_3-characterization}.
		
		Suppose that $\mathcal{M} \not\cong \mathcal{N}$ are structures from $\mathfrak{C}_{st}$. Then there exist indices $i$ and $j$ such that for the structures $\mathcal{M}$ and $\mathcal{N}$, their $P_j$-predicates inside $Q_i$-boxes are not isomorphic. Without loss of generality, one may assume that in this $P_j$-place, $\mathcal{M}$ has order-type $1+\eta$, and $\mathcal{N}$ has order-type $\eta$. Then, it is clear that
		\[
			\mathcal{M}\models \theta_{i,j} \& \neg \psi_{i,j}\ \text{ and } \mathcal{N}\models \neg\theta_{i,j} \& \psi_{i,j}.
		\]
		Therefore, $\mathfrak{C}_{st}$ satisfies property~(b) of the theorem.
	\end{proof}
	
	\smallskip
	
	The rest of the proof for the direction~($\Rightarrow$) is devoted to building a continuous embedding from the given class $\mathfrak{K}$ to $\mathfrak{C}_{st}$. This embedding allows us to apply the Relativized Pullback Theorem (Theorem~\ref{Pullback Theorem}) for finishing our argument.
	
	\medskip
	
	Consider a countable sequence of reals $\vec{\gamma} = (\gamma_i)_{i\in\omega}$.	
	We define an auxiliary continuous operator $\Psi_{\vec{\gamma}}$ as follows. Given a real $\alpha$, our operator $\Psi_{\vec{\gamma}}$ produces a new real $\delta_{\alpha}$, which encodes the atomic diagram of an $L$-structure $\mathcal{S}(\alpha)$.
	
	We always assume that inside $\mathcal{S}(\alpha)$: 
	\begin{itemize}
		\item all predicates $P_i$ are disjoint;
		
		\item every predicate $P_i$ contains at least one element;
		
		\item the domain of $\mathcal{S}(\alpha)$ equals $\omega$.
	\end{itemize}

\subsection*{Construction}
	The construction of $\mathcal{S}(\alpha)$ proceeds in stages. At a stage $s$, for each $i\in\omega$, we define the following auxiliary value:
	\[
		v(i,s) = \begin{cases}
			\min \{ t\leq s\,\colon (\forall x)[ t\leq x \leq s \rightarrow \alpha(x) = \beta_i(x) ]\}, & \text{if } \alpha(s) = \beta_i(s),\\
			\infty, & \text{otherwise}.
		\end{cases}
	\]	
	We also define two parameters $p(s)$ and $b(s)$. Roughly speaking, at a stage $s$, our current ``guess'' is that the input real $\alpha$ is $E_0$-equivalent to $\beta_{p(s)}$, where $p(s) \leq b(s) \leq s$.
	
\subsubsection*{Stage $0$} Put $p(0) = 0$ and $b(s) = 0$. 

\subsubsection*{Stage $s+1$} Consider the following two cases:
	
\subsubsection*{Case~1.} Suppose that there is an index $i \leq s+1$ such that $\alpha(s+1) = \beta_i(s+1)$. 
	
	If $v(p(s), s+1) \neq \infty$, then set $i_0 := p(s)$. Otherwise, $i_0$ is defined as follows.
	\begin{itemize} 
		\item If $p(s) < b(s)$, then $p(s+1) := p(s) + 1$ and $i_0 := p(s) + 1$;
		
		\item If $p(s) = b(s)$, then we define $i_0$ as the least index such that $i_0 \leq s+1$ and 
		\[
			v(i_0,s+1) = \min\{ v(j, s+1) \,\colon j\leq s+1 \}.
		\]
		We set $b(s+1) := s + 1$ and $p(s+1) := i_0$.
	\end{itemize}
	
	Suppose that the relation $P_{i_0}$ (at this particular moment) contains the following linear order:
	$a_0 < a_1 < \dots < a_k$.
	We choose fresh elements $b_0,b_1,\dots,b_k$, add them into $P_{i_0}$, and set:
	\[
		a_0 < b_0 < a_1 < b_1 < \dots < a_k < b_k.
	\]
	
	Consider an index $j\neq i_0$, and suppose that the relation $P_j$ contains the ordering
	$c_0 < c_1 < \dots < c_{\ell}$.
	Choose fresh elements $d_{-1},d_0,d_1,\dots,d_{\ell}$, put them into $P_j$, and define:
	\[
		d_{-1} < c_0 < d_0 < c_1 < d_1 < \dots < c_{\ell} < d_{\ell}.
	\]
	
	\subsubsection*{Case~2} If $\alpha(s) \neq \beta_i(s)$ for all $i \leq s+1$, then for every $j\in\omega$, the relation $P_j$ is arranged in the same way as described in \emph{Case~1}.
	
	This concludes the description of the operator $\Psi_{\vec{\gamma}}$. 
	
	\subsection*{Verification} Similarly to the previous proofs, it is not hard to verify the following properties of $\Psi_{\vec{\gamma}}$:

		\begin{enumerate}
			\item The operator $\Psi_{\vec{\gamma}}$ is $\big( \bigoplus_{i\in\omega} \gamma_i \big)$-com\-pu\-table.
			
			\item If $(\alpha\ E_0\ \gamma_i)$ for some $i\in\omega$, then the structure $\mathcal{S}(\alpha)$ is isomorphic to $\mathcal{S}_i$.
		\end{enumerate}

	Now, let $\Gamma$ be a continuous operator which induces a reduction from $\mathrm{LD}(\mathfrak{K})_{/\cong}$ to $E_3$. For a structure $\mathcal{A}_i$ from $\mathfrak{K}$, fix a real $\beta_i$ such that $\Gamma$ maps all copies of $\mathcal{A}_i$ into the class $[\beta_i]_{E_3}$.

	We define a continuous operator $\Xi$ as follows. Let $\alpha$ be a real.
	\begin{enumerate}
		\item First, we produce the real $\Gamma(\alpha)$.
		
		\item Second, for each $j\in\omega$, we consider the sequence $\vec{\beta}^{[j]} := (\beta_i^{[j]})_{i\in\omega}$. We compute the reals
		\[
			\delta_{\alpha,j} := \Psi_{\vec{\beta}^{[j]}} ( (\Gamma(\alpha))^{[j]}).
		\]
		
		\item Finally, by using the reals $\delta_{\alpha,j}$, $j\in\omega$, we recover a new real $\delta$, which encodes the atomic diagram of an $L_1$-structure $\mathcal{M}$. This structure $\mathcal{M}$ is defined as follows. For each $j$, the $Q_j$-box of $\mathcal{M}$ is an isomorphic copy of the $L_{st}$-structure encoded by $\delta_{\alpha,j}$, and this copy has domain $\{ \langle j, k\rangle\,\colon k \in\omega\}$. We set $\Xi(\alpha) := \delta$.
	\end{enumerate}
	
	It is straightforward to establish the following: the operator $\Xi$ is a continuous embedding from the class $\mathfrak{K}$ into a countable subclass of $\mathfrak{C}_{st}$.
So, by applying the Relativized Pullback Theorem (Theorem~\ref{Pullback Theorem}) to the continuous embedding $\Xi$, we recover a countable family of formulas with the desired properties.  Indeed, the following holds:
\begin{itemize}
 \item by Lemma \ref{lem:C_st_has_a_good_family}, $\mathfrak{C}_{st}$  has a  family of $\Sigma^{\mathrm{inf}}_2$ sentences $\Theta$, which satisfies properties~(a) and~(b) of Theorem~\ref{thm:E_3-characterization};
 \item   by Lemma \ref{continuous is computable with oracle}, $\Phi$ is equivalent to a Turing $X$-operator, for a suitable oracle $X$. 
\end{itemize}
Hence, we can apply Theorem~\ref{Pullback Theorem}, and deduce that $\mathfrak{K}$ has a  family $\Theta^\star$ of $\Sigma^{\mathrm{inf}}_2$ sentences $\Theta$, which satisfies (a) and~(b) of Theorem~\ref{thm:E_3-characterization}, as desired.
	
	\medskip
	
$(\Leftarrow)$.	This direction essentially follows from previous results.  Assume that $\mathfrak{K}:=(\A_i)_{i\in\omega}$ has a family $\Theta$ of $\Sigma_2^{\inf}$ formulas which satisfies the properties (a) and (b) of the theorem. Then it's easy to check that the formulas of $\Theta$ can be arranged to satisfy the following lemma:

\begin{lemma}\label{lem: collection of pairs}
There is a collection of pairs   of formulas $(\rho_{i_0},\rho_{i_1})_{i\in\omega}$ so that, for all structures $\A$ and $\mathcal{B}$ from $\mathfrak{K}$,
\begin{enumerate}
\item $\bigcup_{i\in\omega} \{\rho_{i_0},\rho_{i_1}	\}=\Theta$;
\smallskip

\item for all  $i\in\omega$, $\A$ satisfies exactly one formula between $\rho_{i_0}$ and $\rho_{i_1}$;
\smallskip

\item if $\A\not\cong\mathcal{B}$, then, for some $j\in\omega$,
\[
\A \models \rho_{j_0}  \Leftrightarrow \mathcal{B} \models \rho_{j_1}.
\]
\end{enumerate}
\end{lemma}
The next lemma combines (a limited case of) Theorem~\ref{theorem:characterization learning} with Theorem~\ref{from E_0 to InfEx}.

\begin{lemma}\label{pieces of gamma}
For all $i$, there is a continuous operator $\Gamma_i : 2^\omega \to 2^\omega$ such that, for all structures $\mathcal{S}\in\mathrm{LD}(\mathfrak{K})$,
\begin{itemize}
\item if $\mathcal{S}\models \rho_{i_0}$, then $\Gamma_i(\mathcal{S})\rel{E_0} 0^\infty$;
\item if $\mathcal{S}\models \rho_{i_1}$, then $\Gamma_i(\mathcal{S})\rel{E_0} 1^\infty$.
\end{itemize}
\end{lemma}

\begin{proof}
The proof is similar to that of the direction $(2)\Rightarrow (1)$ of \cite[Theorem~3]{bazhenov2020learning}. Let $i\in\omega$. For $k\in\{0,1\}$, without loss of generality assume that
\[
\rho_{i_k} :=(\exists \bar{x}) \underset{j\in J_{i_k}}{\bigwedge\skipmm{6.8}\bigwedge}\  \forall \bar{y} \phi_{i_k,j}(\bar{x},\bar{y}).
\]
For a finite structure $\mathcal{F}$, say that $\phi_{i_k}$ is \emph{$\mathcal{F}$-compatible} via tuple $\bar{a}$ if within the domain of $\mathcal{F}$ there is no pair $(j,\bar{b})$ with $j\in J_{i_k}$ such that $\mathcal{F}\models \neg \phi_{i_k,j}(\bar{a},\bar{b})$. 

\smallskip

\subsection*{Construction}
Now, let $\alpha$ be a real. Denote by $\mathcal{F}_{\alpha\restriction_s}$ the finite structure (in the signature of $\mathfrak{K}$) encoded by the initial segment $\alpha\restriction_s$ of $\alpha$.
The continuous operator $\Gamma_i$ is defined by stages. 

\subsubsection*{Stage $0$} Let $\Gamma_i(\alpha)(0):=0$ and $\Gamma_i(\alpha)(1):=1$. 

\subsubsection*{Stage s+1} At this stage, we define $\Gamma_i(\alpha)(2s)$ and $\Gamma_i(\alpha)(2s+1)$. To this end, we  distinguish three cases:
\begin{enumerate}
\item There is a tuple $\bar{c}$ so that $\phi_{i_0}$ is $\mathcal{F}_{\alpha \restriction_s}$-compatible via $\bar{c}$, and $\phi_{i_1}$ is not $\mathcal{F}_{\alpha \restriction_s}$-compatible for all tuples $<\bar{c}$. If so, let $\Gamma_i(\alpha)(2s)=\Gamma_i(\alpha)(2s+1):=0$;
\item There is a tuple $\bar{c}$ so that $\phi_{i_1}$ is $\mathcal{F}_{\alpha \restriction_s}$-compatible via $\bar{c}$, and $\phi_{i_0}$ is not $\mathcal{F}_{\alpha \restriction_s}$-compatible for all tuples $\leq \bar{c}$. If so, let $\Gamma_i(\alpha)(2s)=\Gamma_i(\alpha)(2s+1):=1$;
\item If neither of the above cases hold, then let $\Gamma_i(\alpha)(2s):=0$ and $\Gamma_i(\alpha)(2s+1):=1$.
\end{enumerate}

\subsection*{Verification} The continuity of $\Gamma_i$ immediately follows from the construction. Next, suppose that $\beta\in 2^\omega$ encodes a copy of a structure $\mathcal{S}\in \mathfrak{K}$. By Lemma 
\ref{lem: collection of pairs},  $\mathcal{S}$ satisfies exactly one formula between $\phi_{i_0}$ and $\phi_{i_1}$; without loss of generality, assume that $\mathcal{S}\models \phi_{i_1}$.   This means that there is a tuple $\bar{c}$ and a stage $t_0$ so that $\phi_{i_1}$ is $\mathcal{F}_{\beta \restriction_t}$-compatible via $\bar{c}$, for all $t\geq t_0$. On the other hand, since $\mathcal{S}\not\models \phi_{i_0}$, it must be the case that for all tuples $\bar{d}$ (and, in particular, all tuples $\leq\bar{c}$), there must be a stage $t_1$ so that, for all $t\geq t_1$, $\phi_{i_0}$ is not $\mathcal{F}_{\beta \restriction_t}$-compatible. So, for all sufficiently large $x$, $\Gamma_i(\beta)(x)$ is defined by performing action $(2)$ above.  Thus, $\Gamma_i (\beta)$ is $E_0$-equivalent to $1^\infty$, as desired.
 \end{proof}

We can now construct a continuous reduction from $LD(\mathfrak{K})_{/\cong}$ to $E_3$ by merging  the operators $\Gamma_i$'s as follows:
\[
\Gamma(\alpha)(\langle i, x \rangle):=\Gamma_i(\alpha)(x).
\]
It is an easy consequence of Lemma~\ref{pieces of gamma} that, if $\beta_0$ and $\beta_1$ are copies of the same structure $\mathcal{S}\in \mathfrak{K}$, then $\Gamma(\beta_0)\mathrel{E_3} \Gamma(\beta_1)$.   
To deduce that $\Gamma$   is the desired reduction, suppose that  $\beta_0$ and $\beta_1$ are copies of nonisomorphic  structures  $\A$ and $\mathcal{B}$ from $\mathfrak{K}$. By Lemma~\ref{lem: collection of pairs}, there are $j\in\omega$ and $k\in\{0,1\}$ so that $\A \models \phi_{i_k}$ and $\mathcal{B}\models \phi_{i_{1-k}} $. But then, by Lemma~\ref{pieces of gamma}, it follows that $\Gamma(\beta_0)$ and $\Gamma(\beta_1)$ differ on the $j$th column, that is,
\[
{\Gamma(\beta_0)}^{[j]} \mathrel{E_0} k^{\infty} \mbox{ but }  {\Gamma(\beta_1)}^{[j]} \mathrel{E_0} (1-k)^{\infty}.
\]
Thus, $\Gamma(\beta_0) \; \cancel{\mathrel{E_3}} \; \Gamma(\beta_1)$.

\smallskip

This concludes the proof of Theorem~\ref{thm:E_3-characterization}.
\end{proof}


\section{Learning with the help of $Z_0$ and $E_{set}$}
We conclude our examination of the learning power of combinatorial Borel equivalence relations by briefly focusing on two further examples: $Z_0$ and $E_{set}$. Here, the main goal is to finally individuate a Borel equivalence relation which is able to learn a finite family   beyond the reach of our original framework.

\subsection{$Z_0$-learning} 
Before proceeding to a new result, we give a simple useful fact.
Let $\alpha$ and $\beta$ be reals, and let $s\in\omega$. We use the following notation:
\[
	dn(\alpha,\beta;s) = \frac{\mathrm{card}(\{ i\leq s\,\colon \alpha\triangle\beta(i) = 1\})}{s+1}.
\] 
Recall that the equivalence relation $Z_0$ is given by
\[
(\alpha \mathrel{Z_0} \beta)\Leftrightarrow\lim_{k\to \infty} dn(\alpha,\beta;k)=0
\]

\begin{lemma}\label{lem:Z_0}
	Suppose that $(\alpha\ Z_0\ \beta)$ and $(\alpha\ \cancel{Z_0}\ \gamma)$. Then
	\[
		\limsup \!_s\, dn(\alpha,\gamma;s) = \limsup \!_s\, dn(\beta,\gamma;s).
	\]
\end{lemma}
\begin{proof}
	Let $r := \limsup_s dn(\beta,\gamma;s)$. It is sufficient to show that for any $\varepsilon$ such that $0 < \varepsilon < r$, we have
	\[
		\limsup\!_s\, dn(\alpha,\gamma;s) \geq r-\varepsilon.
	\]
	Define $q := r - \varepsilon$.

	Let $N$ be a non-zero natural number. Fix a number $s_0$ such that $dn(\alpha,\beta;s) < \frac{q}{N}$ for all $s\geq s_0$. 
	
	There exists a sequence $(s_j)_{j\in\omega}$, where $s_0 < s_1 <s_2 < \dots$, such that $dn(\beta,\gamma; s_j) > q$ for all $j$. 
	
	Note that every $s$ satisfies the following:
	\begin{multline*}
		\mathrm{card}(\{ i\leq s\,\colon \alpha \triangle \gamma(i) = 1\}) \geq\\ 
		\geq \mathrm{card}(\{ i\leq s\,\colon \beta \triangle \gamma(i) = 1\}) - \mathrm{card}(\{ i\leq s\,\colon \beta \triangle \alpha(i) = 1\}).
	\end{multline*}
	Hence, we have:
	\[
		dn(\alpha,\gamma;s_j) \geq dn(\beta, \gamma;s_j) - dn(\alpha, \beta; s_j) > q - \frac{q}{N} = q\cdot \frac{N-1}{N}.
	\]
	
	Since $N$ was chosen as an arbitrary natural number, we deduce that for any $\delta > 0$, we have $\limsup_s dn(\alpha,\gamma; s) > q - \delta$. This implies \[
	\limsup \!_s\, dn(\alpha,\gamma; s) \geq q. 
\]
Lemma~\ref{lem:Z_0} is proved.
\end{proof}

\medskip

We show that learnability by finite families cannot distinguish between $E_0$ and $Z_0$:

	\begin{thm}\label{prop:Z_0_finite}
		A finite family $\mathfrak{K}$ is $Z_0$-learnable if and only if $\mathfrak{K}$ is $E_0$-learnable. That is, $Z_0$ and $E_0$ are $\mathrm{Learn}^{<\omega}$-equivalent.
	\end{thm}
	\begin{proof}
		Since $E_0$ is computably reducible to $Z_0$ (see Figure~\ref{fig:benchmark}), every $E_0$-learnable family is also $Z_0$-learnable.
		
		Suppose that $\mathfrak{K} = (\mathcal{A}_i)_{i\in\omega}$ is a $Z_0$-learnable family. Let $\Gamma$ be an operator which induces a continuous reduction from $\mathrm{LD}(\mathfrak{K})$ to $Z_0$. For $i\leq n$, we fix $\beta_i$ such that $\Gamma$ maps all copies of $\mathcal{A}_i$ into $[\beta_i]_{Z_0}$. Notice that the reals $\beta_i$ are pairwise not $E_0$-equivalent.
		
		We fix a positive rational $q_0$ such that 
		\[
			q_0 < \min \{  \limsup \!_s\, dn(\beta_i, \beta_j; s)\,\colon i < j \leq n \}.
		\]
		
		There exist an oracle $X$ and a Turing operator $\Phi$ such that $\Gamma(\alpha) = \Phi^{X\oplus \alpha}$ for all $\alpha \in 2^{\omega}$. 
		
		\smallskip
		
		We define an $(X\oplus \bigoplus_{i\leq n} \beta_i)$-computable operator $\Psi$. Let $\alpha$ be a real.
	For $s\in\omega$, by $\ell[s]$ we denote the greatest number such that for every $x\leq \ell[s]$, the value $\Phi^{(X\oplus \alpha)\upharpoonright s}(x)[s]$ is defined.

	At a stage $s$, for each $i\leq n$, we compute the value
		\[
			m_i[s] := \mathrm{card}( \{ t\leq \ell[s] \,\colon dn(\Phi^{X\oplus\alpha}, \beta_i; t ) > q_0 \}  ).
		\]
		We find the least $j\leq n$ such that
		\[
			m_j[s] = \min \{ m_i[s]\,\colon i\leq n\},
		\]
		and set $\Psi(\alpha)(s) := \beta_j (s)$.
		This concludes the description of the operator $\Psi$. 
		
\smallskip		
		
		Suppose that a real $\alpha$ encodes a copy of a structure $\mathcal{A}_{i_0}$ for some $i_0 \leq n$. Then by Lemma~\ref{lem:Z_0}, we have:
		\[
			\lim \!_s \, dn(\Gamma(\alpha), \beta_{i_0}; s) = 0 \ \text{ and } \limsup \!_s\, dn(\Gamma(\alpha), \beta_i;s) > q_0
		\]
		for all $i \neq i_0$. 
		
		Choose a number $t_0$ such that for all $t\geq t_0$, we have $dn(\Phi^{X\oplus \alpha}, \beta_{i_0}; t) \leq q_0$. Fix a stage $s_0$ with $t_0 \leq \ell[s_0]$. Then for all $s\geq s_0$, we have $m_{i_0}[s] = m_{i_0}[s_0]$. 
		
		On the other hand, it is not hard to show that for every $i\neq i_0$, we have $\lim_s m_i[s] = \infty$. This implies that the real $\Psi(\alpha)$ is $E_0$-equivalent to $\beta_{i_0}$.
		
		We deduce that the operator $\Psi$ provides a continuous reduction from $\mathrm{LD}(\mathfrak{K})$ to $E_0$. Theorem~\ref{prop:Z_0_finite} is proved.
	\end{proof}

It is known that $E_3$ is continuously reducible to $Z_0$ (see Figure~\ref{fig:benchmark}). So, $E_3$ is $\mathrm{Learn}^{\omega}$ reducible to $Z_0$. The next question, which is left open, asks if the converse hold.

\begin{question}
Is there a countable $Z_0$-learnable family, which is not $E_{set}$-learnable?
\end{question}


\subsection{$E_{set}$-learning} A distinctive feature of our learning framework is that there are finite families of structures which are not learnable. This is the case, most notably, of the pair of linear orders $\{\omega, \zeta\}$, where $\zeta$ is the order type of the integers. Such a feature is in sharp contrast with classical paradigms, since, e.g.,  any finite collection of recursive functions
is $\Inf\Ex$-learnable. Yet, we have observed that all Borel equivalence relations so far considered are $\mathrm{Learn}^{<\omega}$-equivalent to $E_0$. So, a question comes naturally: how high in the Borel hierarchy one needs to climb to reach an equivalence relation $E$ which is able to learn a nonlearnable finite family? The next proposition shows that $E_{set}$ suffices.

\begin{prop}
	The family $\{ \omega,\zeta\}$ is $E_{set}$-learnable. 
\end{prop}
\begin{proof}	
	Given a real $\alpha$, which encodes a graph with infinite domain $A\subseteq\omega$, one can effectively recover a list $(a_i)_{i\in\omega}$, which enumerates the set $A$ without repetitions. In addition, the recovery procedure is uniform in $\alpha$.
	
	We define a Turing operator $\Psi$. 	For a real $\alpha$, the output $\Psi(\alpha)$ is constructed as follows. For all $i$ and $s$, we put
	\[
		\Psi(\alpha)(\langle 2i,s\rangle) := \begin{cases}
			0, & \text{if } s < i,\\
			1, & \text{if } s \geq i.
		\end{cases}
	\]
	
	Let $B_s$ be the finite linear order, which is encoded by the finite string $\alpha\upharpoonright s$ (note that $B_s$ can be empty). For $i\in\omega$, consider the element $a_i$ (from the list discussed above). If $a_i\not\in B_s$ or $a_i$ is the $\leq_{B_s}$-least element inside $B_s$, then we set $\Psi(\alpha)(\langle 2i+1,s\rangle) := 0$. Otherwise, set $\Psi(\alpha)(\langle 2i+1,s\rangle) := 1$.
	
	Suppose that a real $\alpha$ encodes a copy of $\mathcal{A} \in \{ \omega, \zeta\}$. If $\mathcal{A}$ is isomorphic to $\zeta$, then it is clear that
	\[
		\{ (\Psi(\alpha))^{[m]}\,\colon m\in\omega \} = \{  0^i 1^{\infty} \,\colon i\in\omega\}.
	\]
	If $\mathcal{A} \cong \omega$, then there is an element $a_{i_0}$, which is $\leq_{\mathcal{A}}$-least. This implies
	\[
		\{ (\Psi(\alpha))^{[m]}\,\colon m\in\omega \} = \{  0^i 1^{\infty} \,\colon i\in\omega\} \cup \{ 0^{\infty}\}.
	\]
	Therefore, we deduce that the family $\{ \omega, \zeta\}$ is $E_{set}$-learnable.
\end{proof}


\section{Conclusions}
The investigation conducted in this paper has been fueled by the discovery of a connection between algorithmic learning theory and descriptive set theory. Namely, we proved that the  task of learning a given family of algebraic structures (up to isomorphism) is   equivalent to the task of constructing a suitable continuous reduction to $E_0$. Then, we carefully analyzed the learning power of a number of well-known benchmark Borel equivalence relations. Our results are collected in Figures~\ref{finite-learn} and \ref{countable-learn}.

We wish to conclude by  mentioning three  research directions that originate from the above results and which look promising:

\begin{enumerate}
\item First, it seems  natural to discuss the learning power of other Borel equivalence relations. There is a wide choice---even if one restricts to a small fragment of the Borel hierarchy, such as the $\mathbf{\Pi}^0_3$ equivalence relations (see \cite{gao2008invariant});  
\item Secondly, it would be nice to obtain  learning theoretic or purely syntactic characterizations for $E_{set}$- and $Z_0$-learnabilities of countable families, along the lines of Theorems \ref{thm:E_3-characterization} and \ref{theorem:characterization learning}.
\item Thirdly, observe that our original framework was inherently limited to the countable case, since the learner had to provide a conjecture (i.e., a finite object) for each isomorphism type of the observed family. But now the concept of $E$-learnability can be naturally applied to families of continuum size. This offers a new research opportunity, probably worth considering.
\end{enumerate}

\medskip

\begin{figure}
	\includegraphics[scale=1]{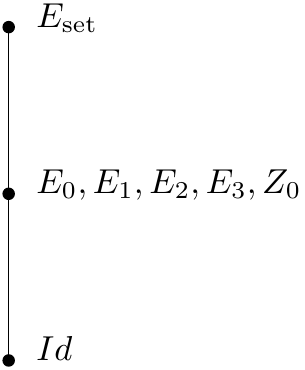}
	\caption{Reductions up to ${\mathrm{Learn}}^{<\omega}$-reducibility.}
	\label{finite-learn}
\end{figure}

\medskip

\begin{figure}
	\includegraphics[scale=1]{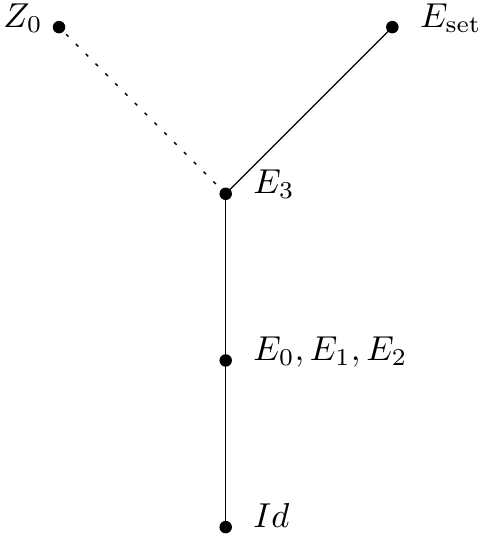}
	\caption{Reductions up to ${\mathrm{Learn}}^\omega$-reducibility.}
	\label{countable-learn}
\end{figure}


\vspace{.15in}


\begin{thebibliography}{BMSMZ21}

\bibitem[AK00]{AK00}
Chris~J. Ash and Julia~F. Knight.
\newblock {\em Computable Structures and the Hyperarithmetical Hierarchy},
  volume 144 of {\em Studies in Logic and the Foundations of Mathematics}.
\newblock Elsevier Science B.V., Amsterdam, 2000.

\bibitem[BFSM20]{bazhenov2020learning}
Nikolay Bazhenov, Ekaterina Fokina, and Luca San~Mauro.
\newblock Learning families of algebraic structures from informant.
\newblock {\em Information and Computation}, 275:104590, 2020.

\bibitem[BMSMZ21]{bazhenov2021computational}
Nikolay Bazhenov, Benoit Monin, Luca San~Mauro, and Rafael Zamora.
\newblock On the computational content of the theory of {B}orel equivalence
  relations.
\newblock {\em Oberwolfach Preprint OWP-2021-06}, 2021.

\bibitem[BSM21]{bazhenov2021turing}
Nikolay Bazhenov and Luca San~Mauro.
\newblock On the {T}uring complexity of learning finite families of algebraic
  structures.
\newblock {\em Journal of Logic and Computation}, 2021.
\newblock Published online. arXiv preprint arXiv:2106.14515.

\bibitem[CCKM04]{CCKM-04}
W.~Calvert, D.~Cummins, J.~F. Knight, and S.~Miller.
\newblock Comparing classes of finite structures.
\newblock {\em Algebra and Logic}, 43(6):374--392, 2004.

\bibitem[CG01]{camerlo2001completeness}
Riccardo Camerlo and Su~Gao.
\newblock The completeness of the isomorphism relation for countable {B}oolean
  algebras.
\newblock {\em Transactions of the American Mathematical Society},
  353(2):491--518, 2001.

\bibitem[CHM12]{coskey2012hierarchy}
Samuel Coskey, Joel~David Hamkins, and Russell Miller.
\newblock The hierarchy of equivalence relations on the natural numbers under
  computable reducibility.
\newblock {\em Computability}, 1(1):15--38, 2012.

\bibitem[EG00]{EG-00}
{Yu}.~L. Ershov and S.~S. Goncharov.
\newblock {\em Constructive models}.
\newblock Kluwer Academic/Plenum Publishers, New York, 2000.

\bibitem[FKSM19]{FKS-19}
Ekaterina Fokina, Timo K{\"o}tzing, and Luca San~Mauro.
\newblock Limit learning equivalence structures.
\newblock In Aur\'elien Garivier and Satyen Kale, editors, {\em Proceedings of
  the 30th International Conference on Algorithmic Learning Theory}, volume~98
  of {\em Proceedings of Machine Learning Research}, pages 383--403, Chicago,
  Illinois, 22--24 Mar 2019. PMLR.

\bibitem[FS89]{FriedmanStanley}
Harvey Friedman and Lee Stanley.
\newblock A {B}orel reducibility theory for classes of countable structures.
\newblock {\em J. Symbolic Logic}, 54(3):894--914, 1989.

\bibitem[Gao09]{gao2008invariant}
Su~Gao.
\newblock {\em Invariant descriptive set theory}.
\newblock CRC Press, Boca Raton, FL, 2009.

\bibitem[Gly85]{Gly:j:85}
Clark Glymour.
\newblock Inductive inference in the limit.
\newblock {\em Erkenntnis}, 22:23--31, 1985.

\bibitem[Gol67]{Gold67}
E.~Mark Gold.
\newblock Language identification in the limit.
\newblock {\em Information and Control}, 10(5):447--474, 1967.

\bibitem[GSWY12]{GaoStephan-12}
Ziyuan Gao, Frank Stephan, Guohua Wu, and Akihiro Yamamoto.
\newblock Learning families of closed sets in matroids.
\newblock In Michael~J. Dinneen, Bakhadyr Khoussainov, and Andr{\'{e}} Nies,
  editors, {\em Computation, Physics and Beyond - International Workshop on
  Theoretical Computer Science, {WTCS} 2012}, volume 7160 of {\em Lecture Notes
  in Computer Science}, pages 120--139, Berlin, 2012. Springer.

\bibitem[Hjo10]{hjorth2010borel}
Greg Hjorth.
\newblock Borel equivalence relations.
\newblock In {\em Handbook of set theory}, pages 297--332. Springer, 2010.

\bibitem[HKL90]{HKL}
L.~A. Harrington, A.~S. Kechris, and A.~Louveau.
\newblock A {G}limm-{E}ffros dichotomy for {B}orel equivalence relations.
\newblock {\em Journal of the American Mathematical Society}, 3(4):903--928,
  1990.

\bibitem[HS07]{HaSt07}
Valentina~S. Harizanov and Frank Stephan.
\newblock On the learnability of vector spaces.
\newblock {\em Journal of Computer and System Sciences}, 73(1):109--122, 2007.

\bibitem[Kan08]{kanoveui}
Vladimir~Grigor'evich Kanove{\u\i}.
\newblock {\em Borel equivalence relations: Structure and classification},
  volume~44.
\newblock American Mathematical Soc., 2008.

\bibitem[KMV07]{KMV07}
Julia~F. Knight, Sara Miller, and Michael {Vanden Boom}.
\newblock Turing computable embeddings.
\newblock {\em J. Symb. Log.}, 72(3):901--918, 2007.

\bibitem[LZZ08]{lange2008learning}
Steffen Lange, Thomas Zeugmann, and Sandra Zilles.
\newblock Learning indexed families of recursive languages from positive data:
  A survey.
\newblock {\em Theoretical Computer Science}, 397(1--3):194--232, 2008.

\bibitem[Mar16]{marker2016lectures}
David Marker.
\newblock {\em Lectures on infinitary model theory}, volume~46 of {\em Lecture
  Notes in Logic}.
\newblock Cambridge University Press, Cambridge, 2016.

\bibitem[Mek81]{mekler1981stability}
Alan~H Mekler.
\newblock Stability of nilpotent groups of class 2 and prime exponent.
\newblock {\em The Journal of Symbolic Logic}, 46(4):781--788, 1981.

\bibitem[Mil21]{miller2021computable}
Russell Miller.
\newblock Computable reducibility for {C}antor space.
\newblock In {\em Structure and Randomness in Computability and Set Theory},
  pages 155--196. World Scientific, Singapore, 2021.

\bibitem[MO98]{Mar-Osh:b:98}
Eric Martin and Daniel Osherson.
\newblock {\em Elements of scientific inquiry}.
\newblock MIT Press, 1998.

\bibitem[MS04]{MS04}
Wolfgang Merkle and Frank Stephan.
\newblock Trees and learning.
\newblock {\em Journal of Computer and System Sciences}, 68(1):134--156, 2004.

\bibitem[Put65]{putnam1965trial}
Hilary Putnam.
\newblock Trial and error predicates and the solution to a problem of
  {M}ostowski.
\newblock {\em The Journal of Symbolic Logic}, 30(1):49--57, 1965.

\bibitem[Soa16]{soare2016turing}
Robert~I. Soare.
\newblock {\em Turing Computability. Theory and Applications}.
\newblock Springer, Berlin, 2016.

\bibitem[SV01]{SV01}
Frank Stephan and Yuri Ventsov.
\newblock Learning algebraic structures from text.
\newblock {\em Theoretical Computer Science}, 268(2):221--273, 2001.

\bibitem[ZZ08]{zz-tcs-08}
Thomas Zeugmann and Sandra Zilles.
\newblock Learning recursive functions: {A} survey.
\newblock {\em Theoretical Computer Science}, 397(1--3):4--56, 2008.

\end{thebibliography}

\end{document}